\documentclass{article}
\usepackage[utf8]{inputenc}
\usepackage{textgreek,mathtools}
\usepackage{amsmath}
\usepackage{graphicx,subfig}
\usepackage[colorinlistoftodos]{todonotes}
\usepackage[colorlinks=true, allcolors=blue]{hyperref}
\usepackage{amsmath}
\usepackage{upgreek}
\usepackage{amssymb}
\usepackage{amsfonts}
\usepackage{amsthm}
\usepackage{mathrsfs}
\usepackage{fontenc}
\usepackage{empheq}
\usepackage{enumerate}
\usepackage{comment}
\usepackage{appendix}
\usepackage{bbm}
\usepackage{algorithm,algpseudocode}
\mathtoolsset{showonlyrefs}

\usepackage{natbib}
\usepackage{multirow}   
\usepackage{booktabs}

\def\R{{\mathbb{R}}}
\def\RR{\R}

\newcommand{\ud}{\,\mathrm{d}}
\newcommand{\MCL}{\mc{L}}
\newcommand{\MCP}{\mc{P}}
\newcommand{\mc}[1]{\mathcal{#1}}
\newcommand{\EE}{\mathbb{E}}
\newcommand{\eps}{\epsilon}

\newtheorem{Theorem}{Theorem}[section]
\newtheorem{Definition}[Theorem]{Definition}

\newtheorem{Assumption}[Theorem]{Assumption}
\newtheorem{Lemma}[Theorem]{Lemma}

\newtheorem{Remark}[Theorem]{Remark}

\def\d{\mathrm{d}}
\newcommand{\cF}{\mathcal{F}}
\newcommand{\cL}{\mathcal{L}}
\newcommand{\cP}{\mathcal{P}}
\newcommand{\cS}{\mathcal{S}}
\newcommand{\cW}{\mathcal{W}}
\newcommand{\cV}{\mathcal{V}}
\newcommand{\cm}{\mathsf{m}}
\newcommand{\QQ}{\mathbb{Q}}
\newcommand{\PP}{\mathbb{P}}
\newcommand{\ZZ}{\mathcal{Z}}
\newcommand{\half}{\frac{1}{2}}

\newcommand{\borig}{B}
\newcommand{\sigmaorig}{\Sigma}
\newcommand{\horig}{H}
\newcommand{\gorig}{G}

\title{Deep Signature Approach for McKean-Vlasov FBSDEs in a Random Environment}
\author{Ruimeng Hu\thanks{Department of Mathematics, and Department of Statistics and Applied Probability, University of California, Santa Barbara, CA, USA, (\href{mailto:rhu@ucsb.edu}{rhu@ucsb.edu})} \and Botao Jin\thanks{Department of Statistics and Applied Probability, University of California, Santa Barbara, CA, USA, (\href{mailto:b_jin@ucsb.edu }{b\_jin@ucsb.edu})} \and  Mathieu Lauri\`ere\thanks{Shanghai Center for Data Science; NYU-ECNU Institute of Mathematical Sciences at NYU Shanghai; NYU Shanghai,  Shanghai, People’s Republic of China, (\href{mailto:mathieu.lauriere@nyu.edu}{mathieu.lauriere@nyu.edu})} \and Jiacheng Zhang\thanks{Department of Statistics, The Chinese University of Hong Kong, New Territories, Hong Kong, 
 (\href{mailto:jiachengzhang@cuhk.edu.hk}{jiachengzhang@cuhk.edu.hk})} }
 \date{}

\begin{document}

\maketitle

\bigskip
\begin{abstract}
   Mean-field games with common noise provide a powerful framework for modeling the collective behavior of large populations subject to shared randomness, such as systemic risk in finance or environmental shocks in economics. These problems can be reformulated as McKean–Vlasov forward–backward stochastic differential equations (MV-FBSDEs) in a random environment, where the coefficients depend on the conditional law of the state given the common noise. Existing numerical methods, however, are largely limited to cases where interactions depend only on expectations or low-order moments, and therefore cannot address the general setting of full distributional dependence.
   
  In this work, we introduce a deep learning–based algorithm for solving MV-FBSDEs with common noise and general mean-field interactions. Building on fictitious play, our method iteratively solves conditional FBSDEs with fixed distributions, where the conditional law is efficiently represented using signatures, and then updates the distribution through supervised learning. Deep neural networks are employed both to solve the conditional FBSDEs and to approximate the distribution-dependent coefficients, enabling scalability to high-dimensional problems. Under suitable assumptions, we establish convergence in terms of the fictitious play iterations, with error controlled by the supervised learning step. Numerical experiments, including a distribution-dependent mean-field game with common noise, demonstrate the effectiveness of the proposed approach.
\end{abstract}

\section{Introduction}

Mean-field game (MFG) theory, introduced by \cite{lasry2007mean} and \cite{huang2006large}, has emerged as a powerful framework to analyze strategic interactions among a large number of rational agents. Each individual has negligible influence, yet the collective distribution of states drives the system’s evolution. The framework has been applied in various fields, including economics, finance, engineering, and the social sciences. A particularly relevant feature in many of these domains is the presence of \emph{common noise}, representing systemic shocks or shared uncertainties such as growth theory~\cite{lasry2008application,gueant2010mean}, systemic risk in financial networks \cite{carmona2015mean}, production of exhaustible resource~\cite{graber2016linear}, crowd motion in uncertain environments~\cite{achdou2018mean}, trading on financial markets~\cite{firoozi2018mean,fu2021mean,bassou2024mean}, bank runs~\cite{carmona2017mean,burzoni2023mean}, electricity and energy transition~\cite{alasseur2020extended,escribe2024mean,dumitrescu2024energy}, climate variability \cite{lavigne2023decarbonization}, price formation~\cite{gomes2023machine}, and macroeconomic fluctuations \cite{vu2025heterogenous}.

The study of MFGs without common noise have been well developed. Solutions can be characterized by the HJB–FP system, coupling a Hamilton–Jacobi–Bellman (HJB) equation for the representative agent and a Fokker–Planck (FP) equation for the distribution dynamics \cite{lasry2007mean}. Alternatively, the problem admits a probabilistic formulation via McKean–Vlasov forward–backward stochastic differential equations (MV-FBSDEs), which encode the optimality and consistency conditions in a pathwise manner \cite{carmona2013probabilistic}. At a more global level, the master equation, a PDE on the space of probability measures, captures the full system, with well-posedness established under certain conditions \cite{cardaliaguet2019master}.

Including common noise in MFGs introduces substantial challenges. The distribution of states is no longer deterministic but evolves as a stochastic process, requiring analysis conditional on the common noise. This leads naturally to stochastic HJB-FP equations or MV-FBSDE in a random environment \cite{CarmonaDelarue_book_II}, both of which are considerably more intricate than their deterministic counterparts. The existing literature has addressed certain classes of such problems, primarily from a theoretical perspective \cite{carmona2016mean,lacker2016general,cardaliaguet2019master, cardaliaguet2022first, lacker2023closed,belak2021continuous,kolokoltsov2019mean,bertucci2023monotone,bensoussan2021linear,huang2022mean,hua2024linear,dianetti2025strong} and modeling applications (see references in the first paragraph). 
On the algorithmic side, contributions remain limited: using a finite difference scheme for a PDE system, \cite{achdou2018mean} studied a crowd motion model with a common noise that can take a finite number of values at a finite number of times; similar models were solved in~\cite{carmona2022convergence} and~\cite{perrin2020fictitious,wu2025population} using respectively deep learning and reinforcement learning; other works that used deep learning for MFG models with common noise are \cite{min2021signatured}, which studied interaction through conditional moments, and \cite{gomes2023machine}, which considered interactions through a one-dimensional stochastic price, and \cite{gu2024global}, which proposed approaches to solve master equations for some MFGs with common noise arising in macro-economics. Yet in many applications, the coefficients depend on the full distribution in nonlinear ways, and developing a rigorous and computationally tractable framework for such settings is quite needed.

This paper aims to fill this gap. We adopt the probabilistic characterization and work with McKean–Vlasov FBSDEs in a random environment:
\begin{equation}
\begin{dcases}
  \ud X_t = \borig(t, \Theta_t, Z_t^0, \MCL(\Theta_t \vert\mc{F}_t^0)) \ud t + \sigmaorig(t, \Theta_t, \MCL(\Theta_t\vert \mc{F}_t^0)) \ud W_t  + \sigmaorig^0(t, \Theta_t, \MCL(\Theta_t \vert \mc{F}_t^0)) \ud W_t^0,   
  \\
  \ud Y_t = -\horig(t,\Theta_t, Z_t^0, \MCL(\Theta_t \vert \mc{F}_t^0)) \ud t + Z_t \ud W_t + Z_t^0 \ud W_t^0, 
  \\ 
  X_0 \sim \mu_0, \qquad Y_T = \gorig(X_T, \MCL(X_T \vert \mc{F}_T^0)), \qquad 
  \Theta_t = (X_t, Y_t, Z_t), \quad t \in [0,T].
\end{dcases}
\end{equation}
Here $\MCL(\cdot \vert \mc{F}^0)$ denotes the marginal law conditional on the common noise filtration, and $(\borig, \sigmaorig,\sigmaorig^0, \horig)$ are functions of compatible dimensions. Importantly, the coefficients depend on the distribution through an embedding \(m\), which provides a compact and flexible way to represent distributional features. This allows us to handle general and nonlinear distributional dependence beyond simple moment-based interactions, while keeping the problem tractable for both analysis and computation. The precise forms of these coefficients are specified in Section~\ref{sec:prelim}. 

Our \textbf{main contribution} lies in developing the \emph{first} framework that integrates signature \cite{lyons2002system,LyonsTerryJ2007DEDb,boedihardjo2014signature,NEURIPS2019_deepsig} with deep-learning based BSDE solvers \cite{han2017deep,han2018solving, hure2020deep,carmona2022convergence,germain2022approximation} within an iterative for MFGs with common noise and general distribution dependence. Our approach shares similarities with the fictitious play algorithm introduced for classical games in~\cite{brown1949some,brown1951iterative} in the sense that players adapt their strategy to the ones of other players; this algorithm has been adapted to MFGs in~\cite{cardaliaguet2017learning} and implemented using deep neural networks and reinforcement learning  in~\cite{lauriere2022scalable,angiuli2025deep,magnino2025solving}. In the present work, instead of averaging over past iterations, we consider that the players react to the last iterate, but we use the terminology of fictitious play consistently with e.g.~\cite{hu2019deep,han2022convergence}. 

Specifically, our approach iteratively solves the above system using Deep BSDE methods \cite{han2018solving}, where the conditional distribution is efficiently encoded by its signature representation, and then updates this representation through supervised learning. This combination enables efficient handling of high-dimensional problems with general nonlinear distribution dependence, going well beyond the conditional-moment structures considered in prior work \cite{min2021signatured}.
We further prove the convergence of this fictitious play scheme: under suitable conditions, the algorithm converges up to a small supervised learning error. Finally, we conduct numerical experiment on the supervised learning step and on a McKean–Vlasov FBSDE in random environment with an explicit solution, as well as construct and test on a flocking model with common noise, 
which provides a nonlinear distribution-dependent benchmark to evaluate performance and demonstrate the accuracy of our approach.

The rest of the paper is organized as follows. Section~\ref{sec:prelim} introduces the probabilistic representation of MFGs with common noise and general distribution dependence, formulated as a McKean-Vlasov FBSDE system. Section~\ref{sec:algo} presents our numerical algorithm, and its convergence is established in Section~\ref{sec:analysis}. Section~\ref{sec:numerics} illustrates the approach through three numerical experiments: a supervising learning example for the measure embedding $m$, an MV-FBSDE with an analytic benchmark solution that does not arise from the linear–quadratic problems, and a mean-field flocking model with common noise. All three involve nonlinear distribution dependence beyond conditional moments.

\section{Preliminaries}\label{sec:prelim}

{\bf General notation. } 
For a dimension $n$, we denote by $\MCP^2(\RR^n)$ the set of probability measures on $\RR^n$ with a second moment. To quantify the distance between probability distributions in $\MCP^2(\RR^n)$, we use the 2-Wasserstein distance $\cW_2$, defined as: for every $\mu,\nu \in \MCP^2(\RR^n)$, 
\begin{equation}\label{def:W2}
    \mathcal{W}_2^2(\mu, \nu) := \inf_{\pi \in \Pi(\mu, \nu)} \int_{\RR^n \times \mathbb{R}^n} |x - y|^2  \pi(dx, dy),
\end{equation}
where $\Pi(\mu, \nu)$ denotes the set of all couplings $\pi$ of $\mu$ and $\nu$,  and $\pi(\d x,\R^n)=\mu(\d x)$ and $\pi(\R^n,\d y)=\nu(\d y)$.
For a random variable $X$ and a sigma-field $\mathcal{G}$, we denote by $\MCL(X)$ the law of $X$ and by $\MCL(X \vert \mathcal{G})$ the conditional law of $X$ given $\mathcal{G}$. 
Let $T$ be a time horizon. Let $d$ be a positive integer for the dimension of the forward and backward variable. Let $q$ be a positive integer for the dimension for the idiosyncratic and common noises (for simplicity, we assume they have the same dimension). We will use $\ell$ to denote the dimension of the measure embedding (see Assumption~\ref{assumption:measure} for details). 
Let $W = (W_t)_{t \ge 0}$ and $W^0 = (W^0_t)_{t \ge 0}$ be two $q$-dimensional Brownian motions, and let $\mc{F} = (\mc{F}_t)_{t \ge 0}$ and $\mc{F}^0= (\mc{F}_t^0)_{t \ge 0}$ be their respective natural filtrations.

\bigskip

\subsection{Definition of the problem}

We consider an FBSDE system whose solution is denoted by $(X,Y,Z,Z^0) = (X_t,Y_t,Z_t,Z^0_t)_{t \in [0,T]}$, where for every $t$ in $[0,T]$, $(X_t,Y_t,Z_t,Z^0) \in \RR^d \times \RR^d \times \RR^{d \times q} \times \RR^{d \times q}$. For simplicity, denote $\Theta_t = (X_t, Y_t, Z_t) \in \mathbb{R}^\theta$, with $\RR^\theta = \RR^d \times \RR^d \times \RR^{d \times q}$. Let $\borig: [0,T] \times \RR^\theta \times \RR^{d \times q}\times \MCP^2(\RR^\theta) \to \RR^d,$ $\sigmaorig: [0,T] \times \RR^\theta \times \MCP^2(\RR^\theta) \to \RR^{d \times q},$ $\sigmaorig^0: [0,T] \times \RR^\theta \times \MCP^2(\RR^\theta) \to \RR^{d \times q}$ denote respectively the drift, the idiosyncratic noise volatility and the common noise volatility in the forward dynamics. Let $\horig: [0,T] \times \RR^\theta\times \RR^{d \times q} \times \MCP^2(\RR^\theta) \to \RR^d$ and $\gorig: \RR^d \times \MCP^2(\RR^d) \to \RR^d$ denote respectively the driver and the terminal condition of the backward dynamics. 
We consider the following FBSDE:
\begin{equation}\label{def:MV-FBSDEorig}
\begin{dcases}
  \ud X_t = \borig(t, \Theta_t, Z_t^0, \MCL(\Theta_t \vert\mc{F}_t^0)) \ud t + \sigmaorig(t, \Theta_t, \MCL(\Theta_t\vert \mc{F}_t^0)) \ud W_t  + \sigmaorig^0(t, \Theta_t, \MCL(\Theta_t \vert \mc{F}_t^0)) \ud W_t^0,   
  \\
  \ud Y_t = -\horig(t,\Theta_t, Z_t^0, \MCL(\Theta_t \vert \mc{F}_t^0)) \ud t + Z_t \ud W_t + Z_t^0 \ud W_t^0, 
  \\ 
  X_0 \sim \mu_0, \qquad Y_T = \gorig(X_T, \MCL(X_T \vert \mc{F}_T^0)), \qquad 
  \Theta_t = (X_t, Y_t, Z_t), \quad t \in [0,T].
\end{dcases}
\end{equation}

\begin{Remark}
    Our study is motivated by FBSDEs stemming from Pontryagin maximum principle, which explains why we consider that $Y$ and $X$ have the same dimension $d$. However, our approach can be adapted directly to situations where they have different dimensions.
\end{Remark}

\begin{Remark}[Existence of solutions and relation to MFG]
Assume that $\sigmaorig, \sigmaorig^0$ do not depend on $(Y_t,Z_t)$, and $\sigmaorig, \sigmaorig^0, \borig, \horig$ depend only on $\cL(X_t\vert\cF_t^0)$ instead of $ \cL(\Theta_t\vert\cF_t^0)$. Then the system~\eqref{def:MV-FBSDEorig} can be interpreted as the MV-FBSDE associated to a master field $\mathcal{U}$ through the relation $Y_t = \mathcal{U}(t, X_t, \cL(X_t \vert \cF_t^0))$, as explained in~\cite[Chapter 5]{CarmonaDelarue_book_II}. Under suitable conditions, Theorem 5.4 gives existence of a solution to the MV-FBSDE in short time.
In the context of MFGs, the master field $\mathcal{U}$ is the solution to the master equation, and the MV-FBSDE corresponds to the equilibrium conditions of an MFG, where $X$ denotes the state of a representative player and $Y$ denotes the value function of this player (in this case, $Y$ is one-dimensional). Alternatively, $Y$ could represent the derivative of the value function, as explained in~\cite[Section 5.4]{CarmonaDelarue_book_II} based on a form of the maximum principle.

\end{Remark}

In the sequel, we suppose that the functions $(\borig, \sigmaorig, \sigmaorig^0, \horig, \gorig)$ depend only on an $\ell$-dimensional embedding of the distribution. More precisely, we assume that they have the following structure. 

\begin{Assumption}[Measure dependence]\label{assumption:measure}
We assume that there exists 
\begin{itemize}
	\item functions $b: [0,T] \times \RR^\theta \times \RR^\ell \to \RR^d,$ $\sigma: [0,T] \times \RR^\theta \times \RR^\ell \to \RR^{d \times q},$ $\sigma^0: [0,T] \times \RR^\theta \times \RR^\ell \to \RR^{d \times q}$, $h: [0,T] \times \RR^\theta \times \RR^\ell \to \RR^d$ and $g: \RR^d \times \RR^\ell \to \RR^d$, 
	\item and functions $m_i: [0,T] \times \mathbb{R}^\theta \times \MCP^2(\mathbb{R}^\theta) \to \mathbb{R}^\ell$,  $i=1,\dots,4$ and $m_5: \mathbb{R}^d \times \MCP^2(\mathbb{R}^d) \to \mathbb{R}^\ell$. 
\end{itemize}
such that, for every $(x,y,z,z^0) \in \RR^\theta \times \RR^{d \times q}$, every $\nu \in \MCP^2(\RR^\theta)$ and every $\mu \in \MCP^2(\RR^d)$,  
\begin{equation}
    \begin{aligned}
       \borig(t, x, y, z, z^0, \nu) 
       &= b(t, x, y, z, z^0, m_1(t, x, y, z, \nu)), 
       \\
       \sigmaorig(t, x, y, z, \nu) 
       &= \sigma(t, x, y, z, m_2(t, x, y, z, \nu)),
       \\
       \sigmaorig^0(t, x, y, z, \nu) 
       &= \sigma^0(t, x, y, z, m_3(t, x, y, z, \nu)),   
       \\
     \horig(t, x, y, z, z^0, \nu) 
     &= h(t, x, y, z, z^0, m_4(t, x, y, z, \nu)),
     \\
     \gorig(x, \mu) 
     &= g(x, m_5(x, \mu)). 
    \end{aligned}
\end{equation}
\end{Assumption}

Based on the above assumption on the form of the coefficients, we can rewrite the FBSDE as follows:
\begin{equation}\label{def:MV-FBSDE-with-m}
\begin{dcases}
  \ud X_t = b(t, \Theta_t, Z_t^0, m_1(t, \Theta_t, \MCL(\Theta_t\vert \mc{F}_t^0))) \ud t + \sigma(t, \Theta_t, m_2(t, \Theta_t, \MCL(\Theta_t\vert \mc{F}_t^0))) \ud W_t \\
  \qquad \qquad + \sigma^0(t, \Theta_t, m_3(t, \Theta_t, \MCL(\Theta_t\vert \mc{F}_t^0))) \ud W_t^0, 
  \\
  \ud Y_t = -h(t,\Theta_t, Z_t^0, m_4(t, \Theta_t, \MCL(\Theta_t\vert \mc{F}_t^0))) \ud t + Z_t \ud W_t + Z_t^0 \ud W_t^0,
  \\ 
  X_0 \sim \mu_0, \qquad Y_T = g(X_T, m_5(X_T, \MCL(X_T\vert \mc{F}_T^0))), \qquad 
  \Theta_t = (X_t, Y_t, Z_t), \quad t \in [0,T].
\end{dcases}
\end{equation}

\subsection{Signatures of paths}\label{subsec:signature}

As seen above, we consider functions $m_i$ of probability distributions. However, for the sake of numerical implementation, we will need to replace laws of stochastic processes by their finite-dimensional approximations. To this end, we will employ path signature \cite{lyons2002system,LyonsTerryJ2007DEDb} and log-signature \cite{liao2019learning}, a powerful tool from rough path theory. 

Denote $\cV^p([0,T],\R^d)$ the space
of continuous mappings from $[0,T]$ to $\R^d$ with finite $p$-variation and equip the space $\cV^p([0,T],\R^d)$, equipped with the norm
$$
\|\cdot\|_{\cV^p} \coloneqq \|\cdot\|_\infty+\|\cdot\|_p,$$
where $\|\cdot\|_\infty$ is the supreme norm and $\|\cdot\|_p$ is defined by
$$
\|X\|_p\coloneqq\bigg(\sup_{D\subset[0,T]}\sum_{i=0}^{r-1}\big\|X_{t_{i+1}}-X_{t_i}\big\|^p\bigg)^{1/p},
$$
with the supremum taken over all partitions $D\subset[0,T]$, $D = \{0\leq t_0 <t_1 <\cdots<t_r\leq T\}$.
\begin{Definition}[Signature]
    Let $X\in\cV^p([0,T],\R^d)$ such that the iterated integrals below are well defined. The signature of $X$, denoted by $\mc{S}(X)$, is an element of $\oplus_{k=0}^\infty(\R^d)^{\otimes k}$ defined by
    $\mathcal{S}(X) = (1,X^1,\cdots,X^k,\cdots),$ 
    with
    \begin{align}\label{def:sig}x^k\coloneqq \int_{0<t_1<\cdots t_k<T}\d x_{t_1} \otimes\cdots\otimes \d x_{t_k}.\end{align}
   For $M \geq 1$, the truncated signature of depth $M$ is  $\mathcal{S}^M(x) := (1,x^1,\cdots,x^M)$, which has dimension $\frac{d^{M+1}-1}{d-1}$.
\end{Definition}

\begin{Definition} [Log-Signature] \label{def: log-sig} Let $x\in\cV^p([0,T],\R^d)$ be a path with its signature $\mathcal{S}(X) = (1, X^1, \cdots, X^k, \cdots)$ well-defined. The log-signature of $x$, denoted by $\log \mathcal{S}(X)$, is defined as the following element in $\oplus_{k=1}^\infty (\R^d)^{\otimes k}$:
\begin{align} \label{log-sig}
    \log \mathcal{S}(X) := \sum_{n=1}^{\infty} \frac{(-1)^{n+1}}{n}(\mathcal{S}(X)-I)^{\otimes n},
\end{align}
where $I = (1, 0, 0, \cdots)$
is the multiplicative identity in the tensor algebra.
\end{Definition}

In our setting, $X$ will be a semimartingale, so the iterated integrals in \eqref{def:sig} are well defined in both the Stratonovich and Itô senses. Several properties (see \cite[Appendix A]{min2021signatured}) make the signature particularly suitable for our framework:
\begin{itemize}
    \item To guarantee a unique representation via signatures, which are otherwise only unique up to tree-like equivalence, we augment the path with time, i.e., we consider $\hat{X}_t = (t,X_t)$. The signature $\mc{S}(\hat X)$ of this augmented path then uniquely characterizes $X$ (see \cite{boedihardjo2014signature} for detailed discussion).
    \item Due to the factorial decay of its terms, the signature can accurately represent a path using only a low-order truncation (a small value of $M$).
    \item The signature serves as a feature map for sequential data, a role justified by its universality property, which is formally stated in the literature \cite{NEURIPS2019_deepsig} and the reference therein.
\end{itemize}
With these in mind, in our algorithm, we would like to approximate the functions $m_i$ using signature-based representations. Specifically, we seek functions $\cm_i$, $i=1,\dots,$ such that: $m_i\big(t,\Theta_t,\cL(\Theta_t|\cF_t^0)\big) \approx \cm_i\big(t,\Theta_t,\mc{S}^M(\hat W_{[0,t]}^0)\big)$ for $i=1,..,4$, and  $m_5\big(X_T,\cL(X_T|\cF_T^0)\big) \approx \cm_5\big(X_T,\mc{S}^M(\hat W_{[0,T]}^0)\big)$, with the option of replacing $\mathcal{S}^M$ by its log version $\log \mathcal{S}^M$.

\section{A Deep Learning Algorithm}\label{sec:algo}

The main idea of our algorithm integrates path signatures with deep learning-based BSDE solvers within a fictitious play framework to solve the system of Mean-Field FBSDEs given in \eqref{def:MV-FBSDE-with-m}. 
The algorithm iteratively solves system~\eqref{def:MV-FBSDE-with-m}. Each iteration $k$ consists of three main steps. Based on the outputs from iteration $k-1$, the $k^{th}$ iteration proceed as follows. 
\begin{itemize}
    \item[\bf Step 1.] Simulate both the forward process $X_t$ and the backward components $(Y_t, Z_t, Z_t^0)$ in discrete time from $0$ to $T$. More precisely, if one works with $N_2$ common Brownian paths $\{W^{0,n_2}\}_{n_2=1}^{N_2}$ and  $N_1$ idiosyncratic Brownian paths $\{W^{n_1,n_2}\}_{n_1=1}^{N_1}$, then for each $W^{0,n_2}$, one needs to simulate $N_1$ paths $\{X^{n_1,n_2}\}$ and the backward  terms $(Y^{n_1,n_2}, Z^{n_1,n_2}, Z^{0,n_1,n_2})$. The continuous analogue of the dynamic of $X^{n_1,n_2}$ and $(Y^{n_1,n_2})$ is given by 
    \begin{equation}\label{def:MV-FBSDE-alg}
\begin{dcases}
  \ud X^{n_1,n_2}_t = b\Big(t, \Theta^{n_1,n_2}_t, Z_t^{0,n_1,n_2}, \cm_1\big(t,\Theta^{n_1,n_2}_t,\mc{S}^M(\hat W_{[0,t]}^{0,n_2})\big)\Big) \ud t \\
  \qquad \qquad + \sigma\Big(t, \Theta^{n_1,n_2}_t, \cm_2\big(t,\Theta^{n_1,n_2}_t,\mc{S}^M(\hat W_{[0,t]}^{0,n_2})\big)\Big) \ud W^{n_1,n_2}_t \\
  \qquad \qquad + \sigma^0\Big(t, \Theta^{n_1,n_2}_t, \cm_3\big(t,\Theta^{n_1,n_2}_t,\mc{S}^M(\hat W_{[0,t]}^{0,n_2})\big)\Big) \ud W_t^{0,n_2}, 
  \\
  \ud Y_t = -h\Big(t, \Theta^{n_1,n_2}_t, Z_t^{0,n_1,n_2}, \cm_4\big(t,\Theta^{n_1,n_2}_t,\mc{S}^M(\hat W_{[0,t]}^{0,n_2}))\big)\Big) \ud t + Z^{n_1,n_2}_t \ud W^{n_1,n_2}_t + Z_t^{0,n_1,n_2} \ud W_t^{0,n_2},
  \\ 
  Y^{n_1,n_2}_0=u\big(X_0^{n_1,n_2}\big),\;Z_t^{n_1,n_2}=v\Big(t,X_t^{n_1,n_2},\mc{S}^M(\hat W_{[0,t]}^{0,n_2})\Big),\;Z_t^{0,n_1,n_2}=v^0\Big(t,X_t^{n_1,n_2},\mc{S}^M(\hat W_{[0,t]}^{0,n_2})\Big),
\end{dcases}
\end{equation}
    which is slightly different from \eqref{def:MV-FBSDE-with-m} and we compute $\mu_T^{n_2}$ and $\nu_t^{n_2}$ as the  empirical versions of the conditional laws of $X_t$ and $\Theta_t = (X_t, Y_t, Z_t)$, conditioned on simulated common noise paths $W^{0,n_2}$.
    \begin{Remark}
        Specifically, we are utilizing  $\cm_i$ to approximate $m_i$ and the random decoupling field $u(X_0)$, $v(X_t)$ and $v^0(X_t)$ to approximate $Y_t,Z_t$ and $Z_t^0$. We have explained our rationale for using $\cm_i$ to approximate $m_i$ in Section \ref{subsec:signature}, primarily due to the universality proposition of the signature. The idea of using decoupling field to represent the backward components $Y_t$, $Z_t$ and $Z_t^0$ stems from the connection between the BSDEs and associated PDEs provided that the Hamiltonian has nice regularity and has been widely used in the literature \cite{han2022convergence}. In this context, we are employing a random decoupling field that incorporates a component of the signature term to account for the randomness arising from the common noise. This results in the form of  $v(t, X_t,\mc{S}^M(\hat W_{[0,t]}^0))$ and $v^0(t, X_t,\mc{S}^M(\hat W_{[0,t]}^0))$.   
    \end{Remark}
    \item[\bf Step 2.] Evaluate $m_i, i=1,\dots,5$ using these empirical measures and set them as training targets for supervised learning. The functions $\cm_i^k$ are then fitted to approximate $m_i$, with inputs consisting of time, the state variables, and the truncated signature of the simulated common noise paths.
    \item[\bf Step 3.] Replace $m_i$ in system~\eqref{def:MV-FBSDE-with-m} with the updated $\cm_i^k$, and solve the resulting system using the Deep BSDE method. This produces updated functions $u^k, v^k, v^{0,k}$, parameterized by neural networks, which approximate $Y_0, Z_t, Z_t^0$ at the $k^{th}$ iteration.
\end{itemize}

\subsection{The detailed algorithm}
\label{sec:detailed-algo}
We describe the algorithms for each step in detail in this subsection.
\begin{itemize}
    \item[\bf Step 0.]\textbf{Initialization:} At the beginning of the algorithm, we need these initialization:
\begin{itemize}
    \item $N_1$: The number of sample paths of the individual noise $W$ generated for each fixed common noise path $W^0$;
    \item $N_2$: The number of sample paths of the common noise $W^0$;
    \item A partition $\pi$ on the interval $[0,T]$ of size $N_T$: $0 = t_0 < t_1 < \cdots < t_{N_T} = T$;
    \item The level of signature $M$ with the corresponding dimension of the truncated signature $d_\text{sig}=\frac{(q+1)^{M+1}-1}{q}$;
    \item The initialization of the decoupling field $u^0:\R^d\to \R^d$ and $v^0,v^{0,0}:[0,T]\times\R^d\times\R^{d_\text{sig}}\to \R^{d\times q}$;
    \item The initial distribution dependence functions $\cm_i^0:[0,T]\times\R^\theta\times\R^{d_{\text{sig}}}\to\R^\ell$ for $i=1,..,4$ and $\cm_5^0:\R^d\times\R^{d_{\text{sig}}}\to\R^\ell$.
\end{itemize} 
We use superscripts $k = 1,2,\dots$ to index iterations. In each iteration stage, the
algorithm performs the following three steps:

\item[\bf Step 1.] {\bf Generate samples.} We generate $N_1\times N_2$ samples of initial condition $\tilde X_0^{k,n_1,n_2}$ and $N_1\times N_2$ samples of individual Brownian motions increments $\Delta W^{k,n_1,n_2}_{t_i}$ and $N_2$ samples of common Brownian motions increment $\Delta W^{0,k,n_2}_{t_i}$ for $n_1 = 1,\cdots,N_1,$ and $n_2 =1,\cdots,N_2$. Recall that  $\mc{S}^M(\hat W^{0,k,n_2}_{[0,t_i]})$ denote the truncated signature (up to level $M$) for the linear interpolation of the augmented path $\hat W^{0,k,n_2}$, over the time interval $[0,t_i]$, computed from the discrete increments $\{\Delta \hat W^{0,k,n_2}_{t_j}\}_{j=0,1,..,i}$. We simulate the forward SDEs in discrete time.  

\begin{equation*}
    \begin{aligned}
        \tilde Y_{0}^{k,n_1,n_2}&=u^{k-1}(\tilde X_{0}^{k,n_1,n_2}),\quad\tilde Z_{t_i}^{k,n_1,n_2}=v^{k-1}\Big(t_i,\tilde X_{t_{i}}^{k,n_1,n_2},\mc{S}^M(\hat W^{0,k,n_2}_{[0,t_i]})\Big),
        \\
        \tilde\Theta_{t_i}^{k,n_1,n_2}&=\big(X_{t_i}^{k,n_1,n_2},Y_{t_i}^{k,n_1,n_2},Z_{t_i}^{k,n_1,n_2}),\quad \tilde Z_{t_i}^{0,k,n_1,n_2}=v^{0,k-1}\Big(t_i,\tilde X_{t_{i}}^{k,n_1,n_2},\mc{S}^M(\hat W^{0,k,n_2}_{[0,t_i]})\Big), 
        \\
        \tilde X_{t_{i+1}}^{k,n_1,n_2}&=\tilde X_{t_{i}}^{k,n_1,n_2}+b\Big(t_i,\tilde\Theta_{t_i}^{k,n_1,n_2}, \tilde Z_{t_i}^{0,k,n_1,n_2}, \cm^{k-1}_1\big(t_i,\tilde\Theta_{t_i}^{k,n_1,n_2},\mc{S}^M(\hat W^{0,k,n_2}_{[0,t_i]})\big)\Big)\Delta t_i
        \\
        &\qquad+\sigma\Big(t_i,\tilde \Theta_{t_{i}}^{k,n_1,n_2},\cm^{k-1}_2\big(t_i,\tilde\Theta_{t_i}^{k,n_1,n_2},\mc{S}^M(\hat W^{0,k,n_2}_{[0,t_i]})\big)\Big)\Delta W^{k,n_1,n_2}_{t_i}
        \\
        &\qquad+\sigma^0\Big(t_i,\tilde  \Theta_{t_{i}}^{k,n_1,n_2},\cm^{k-1}_3\big(t_i,\tilde\Theta_{t_i}^{k,n_1,n_2},\mc{S}^M(\hat W^{0,k,n_2}_{[0,t_i]})\big)\Big)\Delta W^{0,k,n_2}_{t_i},
        \\
        \tilde Y_{t_{i+1}}^{k,n_1,n_2}&=\tilde Y_{t_{i}}^{k,n_1,n_2}-h\Big(t_i,\tilde \Theta_{t_{i}}^{k,n_1,n_2},\cm^{k-1}_4\big(t_i,\tilde\Theta_{t_i}^{k,n_1,n_2},\mc{S}^M(\hat W^{0,k,n_2}_{[0,t_i]})\big)\Big)\Delta t_i
        \\
        &\qquad +\tilde Z_{t_i}^{k,n_1,n_2}\Delta W^{k,n_1,n_2}_{t_i}+Z_{t_i}^{0,k,n_1,n_2}\Delta W^{0,k,n_2}_{t_i}.
    \end{aligned}
\end{equation*}
Define the empirical measures as
\begin{equation*}
    \nu_{t_i}^{k,n_2}:=\frac{1}{N_1}\sum_{n_1=1}^{N_1}\delta_{\tilde\Theta_{t_{i}}^{k,n_1,n_2}},\qquad\mu_{T}^{k,n_2}:=\frac{1}{N_1}\sum_{n_1=1}^{N_1}\delta_{\tilde{X}_{T}^{k,n_1,n_2}}.
\end{equation*}

\item[\bf Step 2.] {\bf Learn the distributions dependence.}
We use data obtained above to construct the supervised learning problem in the second step to approximate the distribution dependence. That is, we have
\begin{equation*}
    \begin{aligned}
    \cm^k_i&:=\text{arginf}_{\cm}\sum_{n_1=1}^{N_1}\sum_{n_2=1}^{N_2}\sum_{i=0}^{N_T-1}\Big\|m_i(t_i,\tilde\Theta_{t_i}^{k,n_1,n_2},\nu_{t_i}^{n_2,k})-\cm\big(t_i,\tilde\Theta_{t_i}^{k,n_1,n_2},\mc{S}^M(\hat W^{0,k,n_2}_{[0,t_i]})\big)\Big\|_2^2,\text{ for }i=1,2,3,4,
    \\
    \cm^k_5&:=\text{arginf}_{\cm}\sum_{n_1=1}^{N_1}\sum_{n_2=1}^{N_2}\Big\|m_5(\tilde X_{T}^{k,n_1,n_2},\mu_{T}^{n_2,k})-\cm\big(\tilde X_{t_{N_T}}^{k,n_1,n_2},\mc{S}^M(\hat W^{0,k,n_2}_{[0,t_{N_T}]})\big)\Big\|_2^2,
    \end{aligned}
\end{equation*}
where $\cm_i$, $i= 1,2,3,4,5$ are searched over a class of functions. One option is to search $\cm$ directly over the neural networks. Another is to restrict $\cm$ to the form
$$
\cm\big(\cdot, \mc{S}^M(\hat W^{0,k,n_2}_{[0,t_i]})\big)=\Big\langle\varphi(\cdot), \mc{S}^M(\hat W^{0,k,n_2}_{[0,t_i]})\big)\Big\rangle,
$$
where $\varphi(\cdot)\in\R^{d_\text{sig}}$is parameterized by a NN. A comparison of these two function classes is given in Section~\ref{sec:numerics1}.

\item[\bf Step 3.] {\bf Use the deep BSDE method to update $u^k$ and $v^{k},v^{0,k}$.} We generate $N$ samples of initial condition $\check X_0^{k,n}$ and $N$ samples of Brownian motions increment  $\Delta \check{W}^{k,n}_{t_i}$ and $\Delta \check{W}^{0,k,n}_{t_i}$ for  $n = 1,\cdots,N$. Again we denote by $S_M\big(\check{W}^{0,k,n}, t_i\big)$ the truncated signature up to level $M$ for path $\check{W}^{0,k,n}$. Then, we update $u^k$ and $v^{k},v^{0,k}$ by
\begin{equation*}
    \begin{aligned}
    u^k,v^{k},v^{0,k}&= \text{arginf}_{u,v,v^0}\sum_{n=1}^N \Big\|g\Big(\check X_T^{k,n},\cm_5^k\big(\tilde X_{t_{N_T}}^{k,n_1,n_2},\mc{S}^M(\hat W^{0,k,n_2}_{[0,t_N]})\big)\Big)-\check Y_T^{k,n}\Big\|_2^2
    \\
    \text{subject to }\check Y_{0}^{k,n}&=u^{k}(\tilde X_{0}^{k,n}),\quad\tilde Z_{t_i}^{k,n}=v^{k}\big(t_i,\tilde X_{t_{i}}^{k,n},\mc{S}^M(\hat W^{0,k,n_2}_{[0,t_i]})\big),
        \\
        \tilde\Theta_{t_i}^{k,n}&=\big(X_{t_i}^{k,n},Y_{t_i}^{k,n},Z_{t_i}^{k,n}),\quad \tilde Z_{t_i}^{0,k,n}=v^{0,k}\big(t_i,\tilde X_{t_{i}}^{k,n},\mc{S}^M(\hat W^{0,k,n_2}_{[0,t_i]})\big), 
        \\
        \check X_{t_{i+1}}^{k,n} &= \check X_{t_{i}}^{k,n}+b\Big(t_i,\check\Theta_{t_i}^{k,n},Z_{t_i}^{0,k,n},\cm_1^k\big(t_i,\tilde\Theta_{t_i}^{k,n_1,n_2},\mc{S}^M(\hat W^{0,k,n_2}_{[0,t_i]})\big)\Big)\Delta t_i
        \\
        &\qquad+\sigma\Big(t_i,\check \Theta_{t_{i}}^{k,n},\cm_2^k\big(t_i,\tilde\Theta_{t_i}^{k,n_1,n_2},\mc{S}^M(\hat W^{0,k,n_2}_{[0,t_i]})\big)\Big)\Delta \check{W}^{k,n}_{t_i}
        \\
        &\qquad+\sigma^0\Big(t_i,\check \Theta_{t_{i}}^{k,n}, \cm^k_3\big(t_i,\tilde\Theta_{t_i}^{k,n_1,n_2},\mc{S}^M(\hat W^{0,k,n_2}_{[0,t_i]})\big)\Big)\Delta \check{W}^{0,k,n}_{t_i},
        \\
        \check Y_{t_{i+1}}^{k,n} &= \check Y_{t_{i}}^{k,n}-h\Big(t_i,\check\Theta_{t_i}^{k,n},\cm^k_4\big(t_i,\tilde\Theta_{t_i}^{k,n_1,n_2},\mc{S}^M(\hat W^{0,k,n_2}_{[0,t_i]})\big)\Big)\Delta t_i
        \\
        &\qquad+Z_{t_i}^{k,n}\Delta \check{W}^{k,n}_{t_i}+Z_{t_i}^{0,k,n}\Delta \check{B}^{k,n}_{t_i}.
    \end{aligned}
\end{equation*}
where $u^k,v^k,v^{0,k}$ are searched over a class of neural networks. Then we iterate the whole procedure to the next round.
\end{itemize}

\section{Convergence analysis}\label{sec:analysis}

In this section, we analyze the convergence of the proposed algorithm. Recall that $\cW_2$ denotes the 2-Wasserstein distance. 

We are going to focus on the following MV-FBSDE in random environment, which is decoupled in the sense that $(Y,Z)$ does not appear in the dynamics of $X$. Furthermore, we assume that $\sigma$ and $\sigma^0$ depends only on $(t,x)$:
\begin{equation}\label{def:MV-FBSDE}
\begin{dcases}
  \ud X_t = b(t, X_t,  m_1(t, X_t, \MCL(X_t\vert \mc{F}_t^0))) \ud t + \sigma(t, X_t) \ud W_t
  + \sigma^0(t, X_t) \ud W_t^0, \quad X_0 \sim \mu_0,  \\
  \ud Y_t = -h(t, \Theta_t, Z_t^0, m_4(t, X_t, \MCL(X_t \vert \mc{F}_t^0))) \ud t + Z_t \ud W_t + Z_t^0 \ud W_t^0, \quad Y_T = g(X_T, m_5(X_T, \MCL(X_T \vert \mc{F}_T^0))),
\end{dcases}
\end{equation}
for some functions $m_1,m_4,m_5$.

Let $\mc{M}_1$ and $\mc{M}_2$ be the function sets:
\begin{align}\label{def:M}
      &\mathcal{M}_1 = \{m:[0,T]\times \RR^d \times \RR^{d_{\text{sig}}} \rightarrow \RR^\ell, \|m(t,x, s) - m(t',x', s)\|^2 \le M \|x - x'\|^2, \quad  \|m(t,x, s)\| \le M[1 + \|x\|]  \},\\
    &\mathcal{M}_2 =\{m:\RR^d \times \RR^{d_{\text{sig}}} \rightarrow \RR^\ell, \|m(x, s) - m(x', s)\|^2 \le M\|x - x'\|^2, \|m(x, s)\| \le M[1+\|x\|]\},
\end{align}
where we recall that $d_{\text{sig}}$ is the dimension of the truncated signature of $\hat W_t := (t, W_t^0)$, i.e, $d_{\text{sig}} = \frac{(q+1)^{M+1}-1}{q}$.
At the $k^{th}$ iteration, assume supervised learning functions $(m_1^{k}, m_4^{k}, m_5^{k}) \in \mc{M}_1 \times \mc{M}_1 \times \mc{M}_2$, we define 
\begin{equation}\label{def:MV-FBSDE-k}
\begin{dcases}
  \ud X_t^k 
  = b(t, X_t^k, m^k_1(t, X_t^k, \mc{S}^M(\hat W^0_{[0,t]}))) \ud t  + \sigma(t, X_t^k) \ud W_t  + \sigma^0(t, X_t^k) \ud W_t^0, \quad X_0^k \sim \mu_0,  
  \\
  \ud Y_t^k = -h(t,\Theta_t^k, Z_t^{0,k}, m^k_4(t, X_t^k, \mc{S}^M(\hat W^0_{[0,t]}))) \ud t + Z_t^k \ud W_t + Z_t^{0,k} \ud W_t^0, \quad Y_T^k = g(X_T^k, m^k_5(X_T^k, \mc{S}^M(\hat W^0_{[0,T]}))).
\end{dcases}
\end{equation}

To prove convergence, we will use the following assumption. 
\begin{Assumption}\label{assumption bsde}
\begin{enumerate}[(a)]
    \item The functions $b: [0, T]\times \RR^d \times \RR^\ell \to \RR^d $, $\sigma,\sigma^0: [0,T] \times \RR^d  \to \RR^{d \times q}$, $h: [0,T] \times \RR^\theta \times \RR^\ell \to \RR^d$ and $g: \RR^d \times \RR^\ell \to \RR^d$ are Lipschitz with respect to all variables except possibly time, with a Lipschitz constant $L$, i.e.: for all $(t,x,x',m,m',z,z') \in [0, T]\times \RR^d \times \RR^d \times \RR^\ell \times \RR^\ell \times \RR^{d \times q} \times \RR^{d \times q}$
    \begin{align*}
        \|b(t,x,m) - & b(t,x',m')\|^2 + \|\sigma(t,x) - \sigma(t,x')\|_F^2 + \|\sigma^0(t,x) - \sigma^0(t,x')\|_F^2  
        \\
         &+ \|h(t,x, y, z, m) - h(t, x', y', z',m')\|^2 + \|g(x,m) - g(x,m')\|^2  \\
        & \qquad \le  L[\|x-x'\|^2 + \|y-y'\|^2 + \|z-z'\|_F^2 + \|m - m'\|^2].
    \end{align*}
   
    \item The functions $m_1:[0,T]\times\RR^d\times\cP^2(\RR^d)\rightarrow \RR^\ell$, $m_4:[0,T]\times\RR^d\times\cP^2(\RR^d)\rightarrow \RR^\ell$ and $m_5:\RR^d \times\cP^2(\RR^d)\rightarrow \RR^\ell$ are Lipschitz with respect to all variables except possibly time, with the same constant $L$, i.e.:
    \begin{multline*}
       \|m_1(t,x,\mu) - m_1(t,x',\mu')\|^2 + \|m_4(t,x,\mu)  - m_4(t,x',\mu')\|^2 \\
       + \|m_5(x,\mu) - m_5(x',\mu')\|^2 
       \le L[\|x - x'\|^2 + \mathcal{W}_2^2(\mu,\mu')].
    \end{multline*}
    
    \item There exists a constant $K$, such that
    \begin{multline*}
        \|b(t,0, 0)\|^2 + \|\sigma(t,0)\|_F^2 + \|\sigma^0(t,0)\|_F^2 +  \|h(t,0,0)\|^2 +  \|g(0,0)\|^2 +  \|m_1(t,0,\delta_0)\|^2 \\
        +\|m_4(t,0,\delta_0)\|^2 +\|m_5(0,\delta_0)\|^2+ \EE\|X_0\|^2 \le K,
    \end{multline*}
    where $\delta_0$ denotes the Dirac measure at $0$.

\end{enumerate}
\end{Assumption}

\begin{Lemma}
\label{lem:bdd-fbsde-sol}
Under Assumption \ref{assumption bsde}, the forward-backward system~\eqref{def:MV-FBSDE} has a unique solution and the forward-backward system~\eqref{def:MV-FBSDE-k} has a unique solution for all $k$. Moreover,
\begin{equation}
    \EE\left[ \sup_{0 \leq t \leq T} \|X_t\|^2 + \|Y_t\|^2 + \int_0^T (\|Z_t\|^2_F + \|Z_t^0\|^2_F) \ud t \right]  \leq C,
\end{equation}
and 
\begin{equation}
    \EE\left[ \sup_{0 \leq t \leq T} \|X^k_t\|^2 + \|Y^k_t\|^2 + \int_0^T (\|Z^k_t\|^2_F + \|Z_t^{0,k}\|^2_F) \ud t \right]  \leq C,
\end{equation}
for all $k\in\mathbb N$ and some constant $C$ depending on $T,K,L$.
\end{Lemma}
\begin{proof}
\noindent{\bf Step 1: Well-posedness of \eqref{def:MV-FBSDE-k}. } 
Observe that $\mathcal{S}^M(\hat W_{[0,t]}^0)$, being the truncated signature up to time $t$, is measurable with respect to $\mc F_t$. Define 
$$
    \tilde b ^k(t,x,\omega)\coloneqq b\Big(t,x,m_1^k\big(t,x,\cS^M(\hat W^0_{[0,t]})(\omega)\big)\Big),\text{ which is }\mathbb{F}\text{-progressively measurable}.
$$
By Assumption \ref{assumption bsde} and the fact that $m_1^k \in \mc{M}_1$, the Lipschitz continuity of $\tilde b^k$ in $x$ can be verified. Therefore, the well-posedness of the forward component $X^k$ follows from \cite[Theorem~3.3.1]{zhang2017backwardbook}. Then, by~\cite[Theorem 3.2.2]{zhang2017backwardbook}, $\{X_t^k\}_{t\in[0,T]}$ is $\mathbb{F}$-progressively measurable and  satisfies
$$
\EE\left[ \sup_{0 \leq t \leq T} \|X_t^k\|^2\right]\leq C,
$$
where the constant $C$ depends on $T,K,L,M$ and the initial distribution but is independent of $k$. Next, to handle the backward component of \eqref{def:MV-FBSDE-k}, define 
$$
    \tilde h^k(t,y,z,z^0,\omega)\coloneqq h\Big(t,X_t^k(\omega),y,z,z^0,m_4^k\big(t,X_t^k(\omega),\mathcal{S}^M(\hat W^0_{[0,t]})(\omega)\big)\Big),
$$
which is again $\mathbb{F}${-progressive measurable}. Note that $X_T^k$ is square integrable and measurable with respect to $\mc F_T=\sigma(W_t, W_t^0, t \leq T)$. By Assumption~\ref{assumption bsde}, the Lipschitz property of $\tilde h^k$ in $(y, z, z^0)$ and the square-integrability condition on $\tilde h^k(0,0,0,0)$ are satisfied. Hence, applying  \cite[Theorem 4.3.1]{zhang2017backwardbook} gives the existence and uniqueness of the BSDE solution $(Y^k,Z^k, Z^{0,k})$, with the estimate 
$$
\EE\left[  \|Y^k_t\|^2 + \int_0^T (\|Z^k_t\|^2_F + \|Z_t^{0,k}\|^2_F) \ud t \right]  \leq C,
$$
where $C$ depends on $T,K,L,M$ and the initial condition but is independent of $k$.

\noindent{\bf Step 2: Well-posedness of \eqref{def:MV-FBSDE}. }
Note that \cite[Theorem 3.2.2]{zhang2017backwardbook} cannot be directly applied due to the presence of the term $\mc L(X_t|\mc F_t^0)$. Nevertheless, we follow a similar strategy, employing Picard iterations to construct a strong solution. We adopt the notation $\mathbb{S}^2(\mathbb{F})$ from \cite{zhang2017backwardbook}, defined as  
$$
    \mathbb{S}^2(\mathbb{F})\coloneqq\Big\{\{X_t\}_{t\in[0,T]}\text{ is }\mathbb{F}\text{-progressively measurable}:\mathbb{E}\big[\sup_{0\leq t\leq T}\|X_t\|^2\big]<\infty\Big\}.
$$
Define a sequence  $\{\tilde X^{n}_t\}_{t\in[0,T]}$ recursively by setting $\tilde X_t^0\equiv X_0$ for $t\in[0,T]$, and for $n\geq 0$
\begin{equation}\label{Xk}
    \ud \tilde X^{n+1}_t = b\Big(t, \tilde X^{n}_t,  m_1\big(t, \tilde X^{n}_t, \MCL(\tilde X^{n}_t\vert \mc{F}_t^0)\big)\Big) \ud t + \sigma(t, \tilde X^{n}_t) \ud W_t
  + \sigma^0(t, \tilde X^{n}_t) \ud W_t^0, \quad \tilde X^{n}_0 = X_0.
\end{equation}
Using Assumption \ref{assumption bsde}, it is straightforward to verify that for all $n \geq 0$
\begin{align}\label{supboundX}
    \EE\Big[\sup_{t\in[0,T]}\|\tilde X^{n}_t\|^2\Big]< \infty.
\end{align}
Indeed, by applying the Burkholder–Davis–Gundy inequality and Assumption~\ref{assumption bsde}, we have
\begin{align*}
    \EE&\Big[\sup_{t\in[0,T]}\|\tilde X^{n+1}_t\|^2\Big]
    \\&\leq C\int_0^T\bigg(\Big\|\bigg(b\Big(t, \tilde X^{n}_t,  m_1\big(t, \tilde X^{n}_t, \MCL(\tilde X^{n}_t\vert \mc{F}_t^0)\big)\Big)\Big\| ^2+\text{Trace}\Big(\sigma(t,\tilde X^{n}_t)\sigma(t,\tilde X^{n}_t)^\top+\sigma^0(t,\tilde X^{n}_t)\sigma^0(t,\tilde X^{n}_t)^\top\Big)\bigg)\ud t
    \\
    &\leq C\Big(1+\EE\Big[\sup_{t\in[0,T]}\|\tilde X^{n}_t\|^2\Big]\Big),
\end{align*}
where the last inequality holds due to Assumption \ref{assumption bsde} and thus inductively confirms \eqref{supboundX}. Define $\Delta X_t^{n}\coloneqq \tilde X^{n}_t-\tilde X^{n-1}_t$, $\Delta \sigma^{n}_t\coloneqq \sigma(t,\tilde X^{n}_t)-\sigma(t,\tilde X^{n-1}_t)$ and $\Delta \sigma_t^{0,n}\coloneqq \sigma^0(t,\tilde X^{n}_t)-\sigma^0(t,\tilde X^{n-1}_t)$ and 
$$
    \Delta b^{n}_t\coloneqq b\Big(t, \tilde X^{n}_t,  m_1\big(t, \tilde X^{n}_t, \MCL(\tilde X^{n}_t\vert \mc{F}_t^0)\big)\Big)-b\Big(t, \tilde X^{n-1}_t,  m_1\big(t, \tilde X^{n-1}_t, \MCL(\tilde X^{n-1}_t\vert \mc{F}_t^0)\big)\Big),
$$
for $k\geq 1$. Then $\Delta X^{n+1}_t$ satisfies the SDE:
$$
    \ud \Delta X^{n+1}_t=\Delta b^{n}_t\ud t +\Delta \sigma^{n}_t\ud W_t+\Delta \sigma^{0,n}_t\ud W^0_t.
$$
Applying It\^o's formula, we obtain
\begin{align*}
    \ud e^{-\lambda t}\big\|\Delta X^{n+1}_t\big\|^2=&e^{-\lambda t}\Big(2\Delta X^{n+1}_t\cdot \Delta b^{n}_t+\text{Trace}\big(\Delta \sigma^{n}_t(\Delta \sigma^{n}_t)^\top+\Delta \sigma^{0,n}_t(\Delta \sigma^{0,n}_t)^\top\big)-\lambda \big\|\Delta X^{n+1}_t\big\|^2\Big)\ud t
    \\
    &+2e^{-\lambda t}\Delta X^{n+1}_t\cdot \big(\Delta \sigma^{n}_t\ud W_t+\Delta \sigma^{0,n}_t\ud W^0_t\big).
\end{align*}
Since $\Delta X^{n+1}_0=0$, and  $e^{-\lambda t}\Delta X^{n+1}_t\cdot \big(\Delta \sigma^{n}_t\ud W_t+\Delta \sigma^{0,n}_t\ud W^0_t\big)$ is a true martingale due to \eqref{supboundX} and \cite[Problem 2.10.7]{zhang2017backwardbook}, we have

\begin{align}\label{itobdd}
\lambda \EE\bigg[\int_0^Te^{-\lambda t}\big\|\Delta X^{n+1}_t\big\|^2\ud t\bigg]\leq \EE\bigg[\int_0^T\bigg(e^{-\lambda t}\Big(2\Delta X^{n+1}_t\cdot \Delta b^{n}_t+\text{Trace}\big(\Delta \sigma^{n}_t(\Delta \sigma^{n}_t)^\top+\Delta \sigma^{0,n}_t(\Delta \sigma^{0,n}_t)^\top\big)\Big)\ud t\bigg)\bigg].
\end{align}
Using Assumption \ref{assumption bsde}, and the definition of $\Delta\sigma^{n}_t$, we derive the bounds
\begin{align}\label{sigmabd}
    \text{Trace}\big(\Delta \sigma^{n}_t(\Delta \sigma^{n}_t)^\top\big)=\|\Delta \sigma^{n}_t\|_F^2=\|\sigma(t,\tilde X^{n}_t)-\sigma(t,\tilde X^{n-1}_t)\|_F^2\leq L\|\Delta X^{n}_t\|^2,
\end{align}
and similarly 
\begin{align}\label{sigma0bd}
    \text{Trace}\big(\Delta \sigma^{0,n}_t(\Delta \sigma^{0,n}_t)^\top\big)\leq  L\|\Delta X^{n}_t\|^2.
\end{align} As for the term of $\Delta b^{n}_t$, we have
\begin{equation}\label{bbd}
\begin{aligned}
    \EE\Big[\big\|&\Delta b^{n}_t\big\|^2\Big]=\EE\bigg[\Big\|b\Big(t, \tilde X^{n}_t,  m_1\big(t, \tilde X^{n}_t, \MCL(\tilde X^{n}_t\vert \mc{F}_t^0)\big)\Big)-b\Big(t, \tilde X^{n-1}_t,  m_1\big(t, \tilde X^{n-1}_t, \MCL(\tilde X^{n-1}_t\vert \mc{F}_t^0)\big)\Big)\Big\|^2\bigg]
    \\
    &\leq \EE\bigg[L\Big(\big\|\tilde X^{n}_t-\tilde X^{n-1}_t\|^2+\Big\|m_1\big(t, \tilde X^{n}_t, \MCL(\tilde X^{n}_t\vert \mc{F}_t^0)\big)-m_1\big(t, \tilde X^{n-1}_t, \MCL(\tilde X^{n-1}_t\vert \mc{F}_t^0)\big)\Big\|^2\Big)\bigg]
    \\
    &\leq \EE\bigg[L\Big(\big\|\tilde X^{n}_t-\tilde X^{n-1}_t\|^2+L\Big(\big\|\tilde X^{n}_t-\tilde X^{n-1}_t\|^2+\mc W_2^2\big(\MCL(\tilde X^{n}_t\vert \mc{F}_t^0),\MCL(\tilde X^{n-1}_t\vert \mc{F}_t^0)\big)\Big)\Big)\bigg]
    \\
    &\leq (2L^2+L)\EE\big[\big\|\Delta X^{n}_t\big\|^2\big].
\end{aligned}
\end{equation}

Combining these estimates with \eqref{itobdd}, we obtain
$$
    \lambda \EE\bigg[\int_0^Te^{-\lambda t}\big\|\Delta X^{n+1}_t\big\|^2\ud t\bigg]\leq \EE\bigg[\int_0^T\bigg(e^{-\lambda t}\Big(\big\|\Delta X^{n+1}_t\big\|^2+\big((2L^2+L)+2L\big)\big\|\Delta X^{n}_t\big\|^2\Big)\ud t\bigg)\bigg].
$$
Set $\lambda\coloneqq 1+4\big(2L^2+3L\big)$. Then, we have the contraction property
$$
    \EE\bigg[\int_0^Te^{-\lambda t}\big\|\Delta X^{n+1}_t\big\|^2\ud t\bigg]\leq \frac14\EE\bigg[\int_0^Te^{-\lambda t}\big\|\Delta X^{n}_t\big\|^2\ud t\bigg].
$$
Following the argument in \cite[Theorem 3.3.1]{zhang2017backwardbook}, there exists a limit $\mathcal{X}\in\mathbb{S}^2(\mathbb{F})$ such that
$$
    \EE\Big[\sup_{0\leq t\leq T}\big\|\tilde X^{n}_t-\mathcal{X}_t\big\|^2\Big]\to 0,\text{ for }n\to\infty,
$$
that is $\tilde X^{n}\to \mathcal{X}$ in $\mathbb{S}^2(\mathbb{F})$. Taking $n \to \infty$ in \eqref{Xk}, we obtain that $\mathcal{X}$ satisfy the following
$$
    \ud \mathcal{X}_t = b\Big(t, \mathcal{X}_t,  m_1\big(t, \mathcal{X}_t, \MCL(\mathcal{X}_t\vert \mc{F}_t^0)\big)\Big) \ud t + \sigma(t, \mathcal{X}_t) \ud W_t
      + \sigma^0(t, \mathcal{X}_t) \ud W_t^0, \quad \mathcal{X}_0 = X_0,
$$
which verify that $\mathcal{X}$ solves the forward equation of \eqref{def:MV-FBSDE}. 

The well-posedness of the backward equation of \eqref{def:MV-FBSDE} follows analogously to the previous case of equation \eqref{def:MV-FBSDE-k}. The uniform square-integrability estimates follow immediately from the above arguments, thus completing the proof.

\end{proof}

We prove the following result, which is an extension of~\cite[Theorem 3.7]{hanhulong2022learning}. 
\begin{Theorem}
\label{thm:main-convergence}
Let $(X,Y,Z,Z^0)$ be the solution to the MV-FBSDE~\eqref{def:MV-FBSDE} and $(X^k,Y^k,Z^k,Z^{0,k})$ be the solution to~\eqref{def:MV-FBSDE-k}. Then, there exist constants $C>0$ and $0<q<1$ depending only on the data of the problem such that:
\begin{align*}
    &\sup_{0 \le t \le T} \left[\EE\|X_t - X^k_t\|^2 + \EE\|Y_t - Y^k_t\|^2 \right] + \int_0^T \left[\EE\|Z_t - Z^k_t\|^2 + \EE\|Z^{0}_t - Z^{0,k}_t\|^2 \right]\ud t
    \\
    \le 
    & \, C \Big\{q^k + \sum_{j=0}^{k-1}q^{k-j} \int_0^T \EE \|m_1^{j+1}(t, X^{j+1}_t, \mc{S}^M(\hat W^0_{[0,t]})) - m_1(t, X^{j+1}_t, \MCL(X_t^j \vert \mc{F}_t^0))\|^2 \ud t 
    \\
    & \qquad\qquad\qquad\qquad + \int_0^T \EE \|m_4^{k}(t, X^{k}_t, \mc{S}^M(\hat W^0_{[0,t]})) - m_4(t, X^{k}_t, \MCL(X_t^{k-1} \vert \mc{F}_t^0))\|^2 \ud t 
    \\
    & \qquad\qquad\qquad\qquad + \EE \|m_5^{k}(X^{k}_T, \mc{S}^M(\hat W^0_{[0,T]})) - m_5(X^{k}_T, \MCL(X_T^{k-1} \vert \mc{F}_T^0))\|^2 \Big\}.
\end{align*}

\end{Theorem}
\begin{proof}
    In the sequel, $C$ denotes a constant depending only on the data of the problem, whose value may change from one line to another. 
    We estimate the difference between $\Theta$ and $\Theta^k$. Let $\delta X^k_t = X_t - X^k_t$, $\delta Y^k_t = Y_t - Y^k_t$, $\delta Z^k_t = Z_t - Z^k_t$, $\delta Z^{0,k}_t = Z^0_t - Z^{0,k}_t$.

    We define:
    $$
        I^k_1 \coloneqq \int_0^T \EE \| m_1(t, X_t^{k}, \MCL(X_t^{k-1}\vert \mc{F}_t^0)) - m^k_1(t, X_t^k, \mc{S}^M(\hat W^0_{[0,t]}))\|^2 \ud t.
    $$
    
    By~\cite[Theorem 3.2.4]{zhang2017backwardbook}, the Lipschitz property of $b$, the definition of $I^k_1$, the assumption on $m_1$ and the definition of the distance $\mathcal{W}_2$, we have: 
    \begin{align*}
        \sup_{0 \le s \le t} \EE \|\delta X^k_s\|^2
        & \le C \int_0^t \EE \|b(s, X_s^k, m_1(s, X_s^k, \MCL(X_s\vert \mc{F}_s^0)))) - b(s, X_s^k, m^k_1(s, X_s^k, \mc{S}^M(\hat W^0_{[0,s]})))\|^2 \ud s
        \\
        & \le C \int_0^t \EE \| m_1(s, X_s^k, \MCL(X_s\vert \mc{F}_s^0))) - m^k_1(s, X_s^k, \mc{S}^M(\hat W^0_{[0,s]})))\|^2 \ud s
        \\
        & \le C \int_0^t \EE \| m_1(s, X_s^k, \MCL(X_s\vert \mc{F}_s^0))) - m_1(s, X_s^{k}, \MCL(X_s^{k-1}\vert \mc{F}_s^0)))\|^2 \ud s + C I^k_1
        \\
        & \le C \int_0^t \EE\left[ \mathcal{W}^2_2(\mc{L}(X_s|\mc{F}^0_s), \mc{L}(X^{k-1}_s|\mc{F}^0_s)) \right] \ud s + C I^k_1
        \\
        & \le C \int_0^t \EE\|\delta X^{k-1}_s \|^2 \ud s + C I^k_1.
    \end{align*}
    
    Then, by induction, one has: 
    \begin{equation}\label{eq:convergence-induction-k}
        \int_0^t \EE\|\delta X^{k}_s \|^2 \ud s \le \frac{C^k}{k!} \int_0^t (t-s)^k \EE \|\delta X^{0}_s \|^2 \ud s + \sum_{j=1}^k \frac{(C t)^j}{j!} I_1^{k+1-j}.
    \end{equation}
    
    From Lemma~\ref{lem:bdd-fbsde-sol}, one can easily deduce that $\EE [\sup_{t \in [0,T]}  \| \delta X^{0}_t \|^2] \leq \EE [\sup_{t \in [0,T]}  \| X_t \|^2] + \EE [\sup_{t \in [0,T]}  \|  X^{0}_t \|^2] \le C$.
    Going back to~\eqref{eq:convergence-induction-k}, we obtain:
    \begin{equation}
    \label{eq:induction-bound-int-deltaXk}
        \int_0^T \EE\|\delta X^{k}_s \|^2 \ud s 
        \le \frac{(CT)^k}{k!} + \sum_{j=1}^k \frac{(C T)^j}{j!} I_1^{k+1-j} 
        \le \tilde{C} \bigg(q^k + \sum_{j=0}^{k-1} q^{k-j} I_1^{j+1}\bigg),
    \end{equation}
    for some constants $\tilde{C}>0$ and $0<q<1$ such that $\frac{C^j}{j!} \le \tilde{C} q^j$ for all $j \in \mathbb{N}^+$. 
    
    We proceed similarly for $(Y,Z,Z^0)$. More precisely, let
    \begin{align}
         &I_4^k = \int_0^T \EE\|m_4(t, X_t^k, \MCL(X_t^{k-1}\vert \mc{F}_t^0)) - m^k_4(t, X_t^k, \mc{S}^M(\hat W^0_{[0,t]}))\|^2 \ud t, \\
         &I_5^k = \EE\|m_5(X_T^k, \MCL(X_T^{k-1}\vert \mc{F}_t^0)) - m^k_5(X_T^k, \mc{S}^M(\hat W^0_{[0,T]}))\|^2.
    \end{align}
    Using~\cite[Theorem 4.2.3]{zhang2017backwardbook}, we have:
    \begin{align*}
        &\sup_{t \le s \le T} \EE \|\delta Y^k_s\|^2 + \EE \int_t^T \|\delta Z^k_s\|^2 \ud s + \EE \int_t^T \|\delta Z^{0,k}_s\|^2 \ud s
        \\
        & \le C \int_t^T \EE \|h(s, X_s, Y_s^k, Z_s^k, Z_s^{0,k}, m_4(s, X_s, \MCL(X_s\vert \mc{F}_s^0)))) - h(s, \Theta_s^k, Z_s^{0,k}, m^k_4(s, X_s^k, \mc{S}^M(\hat W^0_{[0,s]})))\|^2 \ud s \\
        & \qquad + C\EE \|g(X_T, m_5(X_T, \mc{L}(X_T | \mc{F}_T^0)) - g(X_T^k, m_5^k(X_T^k, \mc{S}^M(\hat W^0_{[0,T]}))\|^2
        \\
        & \le C \int_t^T \EE \| \delta X^{k}_s \|^2 + \EE \|m_4(s, X_s^k, \MCL(X_s\vert \mc{F}_s^0)) - m_4(s, X_s^k, \MCL(X_s^{k-1}\vert \mc{F}_s^0))\|^2 \ud s  + C I_4^k \\
        & \qquad + C\EE \| \delta X^{k}_T \|^2 + C\EE\|m_5(X_T^k, \mc{L}(X_T | \mc{F}_T^0)) - m_5(X_T^k, \mc{L}(X_T^{k-1}| \mc{F}_T^0))\|^2 + CI_5^k
        \\
        & \le C \int_t^T \EE [\mathcal{W}_2^2 (\MCL(X_s\vert \mc{F}_s^0),  \MCL(X_s^{k-1}\vert \mc{F}_s^0))] \ud s  + C I_4^k + C\EE[\mathcal{W}_2^2 (\MCL(X_T\vert \mc{F}_T^0),  \MCL(X_T^{k-1}\vert \mc{F}_T^0))] + C I_5^k \\
        & \qquad + C\int_t^T \EE\| \delta X^{k}_s \|^2 \ud s + C \EE \| \delta X^{k}_T \|^2
        \\
        & \le C \int_t^T \EE \| \delta X^{k-1}_s \|^2 \ud s + C I_4^k + C\EE \| \delta X^{k-1}_T\|^2 + CI_5^k   +  C\int_t^T \EE\| \delta X^{k}_s \|^2 \ud s + C \EE \| \delta X^{k}_T \|^2 
        \\
        & \le C \int_0^T \EE \| \delta X^{k-1}_s \|^2 \ud s + C I_4^k + C\int_0^T \EE \| \delta X^{k-2}_s \|^2 \ud s + CI_1^{k-1} + CI_5^k + C I_1^k,
    \end{align*}
    where in the last inequality, we have used 
    $$
    \int_t^T \EE\| \delta X^{k}_s \|^2 \ud s \leq C \EE \| \delta X^{k}_T \|^2 \leq C\int_0^T \EE\|\delta X^{k-1}_s \|^2 \ud s + C I^k_1.
    $$
    
    Using the bound~\eqref{eq:induction-bound-int-deltaXk}, we get:
    \begin{align}
        \sup_{0 \le s \le T} \EE \|\delta Y^k_s\|^2 + \EE \int_0^T \|\delta Z^k_s\|^2 \ud s + \EE \int_0^T \|\delta Z^{0,k}_s\|^2 \ud s
         \leq C q^k + C\sum_{j=0}^{k-1} q^{k-j} I_1^{j+1} +C I_4^k + C I_5^k.
    \end{align}
    Combining the terms, we get:
    \begin{align*}
        &\sup_{0 \le s \le T} \EE \|\delta X^k_s\|^2
        + \sup_{0 \le s \le T} \EE \|\delta Y^k_s\|^2 + \EE \int_0^T \|\delta Z^k_s\|^2 +  \|\delta Z^{0,k}_s\|^2 \ud s 
        \le 
        C \left(q^k + \sum_{j=0}^{k-1} q^{k-j} I_1^{j+1} + I_4^k + I_5^k \right).
    \end{align*}
    
\end{proof}

\begin{Lemma}\label{lem:Girsanov}
    
Assume the function $b$ is of the  form $b(t, x, m) = \sigma(t, x)\phi(t, x, m) +  b^0(t, x)$, such that
\begin{align}\label{assump:phi}
    &\|\sigma(t, x)\|_S \leq K, \quad \|\phi(t, x, m)\|_\infty \leq K, \quad \|b^0(t,x) - b^0(t, x')\|^2 \leq L\|x - x'\|^2, \quad \|b^0(t, 0)\|^2 \leq K, \\
    & \|\phi(t,x, m) - \phi(t, x', m')\|^2 \leq L\left[\|x - x'\|^2 + \|m-m'\|^2\right], 
\end{align}
for all $t \in [0,T]$, 
where $\|\cdot\|_S$ and $\|\cdot \|_\infty$ denote the spectral norm and infinity norm, respectively. 

Let $\bar m^1, \bar m^2 \in \mathcal{M}_1$ be two functions of the form $\bar m^i (t,x, S^M(\hat W^0_{[0,t]}))$, and $\bar X^i$ solve the associated SDE:
\begin{equation}
    \ud \bar X_t^i = b(t, \bar X_t^i, \bar m^i(t, \bar X_t^i, \mathcal{S}^M(\hat W^0_{[0,t]}))) \ud t + \sigma(t, \bar X_t^i) \ud W_t + \sigma^0(t, \bar X_t^i) \ud W_t^0.
\end{equation}
Then for any functions $(x, W^0) \mapsto m(x, W^0)$ and $(x, W^0) \mapsto m'(x, W^0)$ that are at most linearly growing in $x$, and any $\eps \in (0,1)$, there exists a constant $C(\eps)$ such that
\begin{equation}
    \EE \|m(\bar X_t^1, W^0) - m'(\bar X_t^1, W^0)\|^2 \leq C(\eps) \left[\EE\|m(\bar X_t^2, W^0) - m'(\bar X_t^2, W^0)\|^2\right]^\eps, \quad \forall t \in [0,T].
\end{equation}
\end{Lemma}
\begin{proof}

The proof follows the argument in \cite[Theorem 4]{han2022convergence}. For completeness, we provide the details below.

The well-posedness of $\bar X^i$ has been established in Lemma~\ref{lem:bdd-fbsde-sol}.
Let  $\delta \phi_t := \phi(t, \bar X_t^1, \bar m^1(t, \bar X_t^1, \mathcal{S}^M(W_{[0,t]}^0))) -  \phi(t, \bar X_t^1, \bar m^2(t, \bar X_t^1, \mathcal{S}^M(W_{[0,t]}^0)))$ and define the Radon-Nikodym derivative
$$
    \frac{\ud\QQ}{\ud \PP} \equiv \ZZ := \exp\bigg\{-\int_0^T \delta \phi_t \ud W_t - \half \int_0^T |\delta \phi_t|^2\ud t \bigg\}.
$$
By the boundedness of $\phi$, the Novikov condition is satisfied, so $\QQ \sim \PP$. Under $\QQ$, the process $(W_\cdot^0, W_\cdot^\QQ :=  W_\cdot + \int_0^\cdot \delta\phi_s \ud s)$ is a standard Brownian motion, and the distribution of $\bar X^1$ coincides with that of $\bar X^2$ under $\PP$.

Let $\EE_\QQ$ denote the expectation under $\QQ$, and $\EE^p_\QQ$ denotes $(\EE_\QQ[\cdot])^p$. For $\gamma > 2$, we compute:
\begin{align}
    \EE_\QQ[\ZZ^{-\gamma}] & = \EE_\QQ\bigg[\exp\bigg\{\gamma\int_0^T \delta \phi_t \ud W_t^\QQ - \frac{\gamma}{2} \int_0^T |\delta \phi_t|^2\ud t \bigg\}\bigg]\\
    & \leq   \EE_\QQ^{1/2} \bigg[\exp\bigg\{2\gamma\int_0^T \delta \phi_t \ud W_t^\QQ - 2\gamma^2 \int_0^T |\delta \phi_t|^2\ud t \bigg\}\bigg] \times \EE_\QQ^{1/2}\bigg[\exp\bigg\{ (2\gamma^2-\gamma) \int_0^T |\delta \phi_t|^2\ud t \bigg\}\bigg] \\
    & \leq e^{CT(\gamma^2 -\half \gamma)},
\end{align}
where we used the Cauchy–Schwarz inequality, the martingale property, and the boundedness of $\phi$.

Consequently,
\begin{align}
    \EE_\PP &\|m(\bar X_t^1, W^0) - m'(\bar X_t^1, W^0)\|^2 = \EE_\QQ \left[\|m(\bar X_t^1, W^0) - m'(\bar X_t^1, W^0)\|^2 \ZZ^{-1}\right]\\
   & \leq \EE_\QQ^{1-\frac{1}{\gamma}} \left[\|m(\bar X_t^1, W^0) - m'(\bar X_t^1, W^0)\|^{\frac{2\gamma}{\gamma-1}}\right]\EE_\QQ^{\frac{1}{\gamma}}\left[\ZZ^{-\gamma}\right] \\
   & \leq C(\gamma) \EE_\QQ^{1-\frac{2}{\gamma}} \left[\|m(\bar X_t^1, W^0) - m'(\bar X_t^1, W^0)\|^2\right] \EE_\QQ^{\frac{1}{\gamma}} \left[\|m(\bar X_t^1, W^0) - m'(\bar X_t^1, W^0)\|^4\right] \\
   & \leq C(\gamma) \EE_\PP^{1-\frac{2}{\gamma}} \left[\|m(\bar X_t^2, W^0) - m'(\bar X_t^2, W^0)\|^2\right],
\end{align}
where we consecutively used H\"{o}lder's inequality, the bound on $\EE_\QQ[\mc{Z}^{-\gamma}]$, H\"{o}lder's inequality again, and the linear growth condition of $m, m'$, and \cite[Theorem 3.4.3]{zhang2017backwardbook}. 
Here $C(\gamma)$ is a constant depending on $T$, $L$, $K$ and $\gamma$, and may vary from line to line. Noting that  $0 < 1-\frac{2}{\gamma} < 1$, the proof is complete.
\end{proof}

\begin{Theorem}

    Let $(X,Y,Z,Z^0)$ be the solution to the MKV FBSDE~\eqref{def:MV-FBSDE} and $(X^k,Y^k,Z^k,Z^{0,k})$ be the solution to~\eqref{def:MV-FBSDE-k}. Then, there exist constants $\eps \in (0,1)$, $C>0$ and $0<q<1$ depending on the data of the problem and $\eps$ such that:
\begin{align*}
    &\sup_{0 \le t \le T} \left[\EE\|X_t - X^k_t\|^2 + \EE\|Y_t - Y^k_t\|^2 \right] + \int_0^T \left[\EE\|Z_t - Z^k_t\|^2 + \EE\|Z^{0}_t - Z^{0,k}_t\|^2 \right]\ud t
    \\
    \le 
    & \, C(\eps) \Big\{q^k + \sum_{j=0}^{k-1}q^{k-j}  \int_0^T \EE \|m_1^{j+1}(t, X^{j}_t, \mc{S}^M(\hat W^0_{[0,t]})) - m_1(t, X^{j}_t, \MCL(X_t^j \vert \mc{F}_t^0))\|^2 \ud t 
    \\
    & \qquad\qquad\qquad\qquad + \int_0^T \EE \|m_4^{k}(t, X^{k-1}_t, \mc{S}^M(\hat W^0_{[0,t]})) - m_4(t, X^{k-1}_t, \MCL(X_t^{k-1} \vert \mc{F}_t^0))\|^2 \ud t
    \\
    & \qquad\qquad\qquad\qquad + \EE \|m_5^{k}(X^{k-1}_T, \mc{S}^M(\hat W^0_{[0,T]})) - m_5(X^{k-1}_T, \MCL(X_T^{k-1} \vert \mc{F}_T^0))\|^2\Big\}^\eps.
\end{align*}
\end{Theorem}
\begin{proof}

    Since both $m_i$ and $m^j_i$ are Lipschitz in $x$, one can apply Lemma~\ref{lem:Girsanov}, identifying $m^{j+1}_i$ and $m^j_i$ with $\bar m^1$ and $\bar m^2$, and $m^j_i$ and $m_i$ as $m$ and $m'$ as defined in the lemma. Combined with Theorem~\ref{thm:main-convergence}, this yields the desired result.
\end{proof}

\section{Numerical Experiments}\label{sec:numerics}
This section is dedicated to evaluating the numerical performance of our proposed algorithm. The implementation is carried out in Python using \texttt{Pytorch} and \texttt{signatory}, and the code is available upon request. Recall that $W_t$ and $W_t^0$ are independent standard $q$-dimensional Brownian motions, and $X_t$ is the $d$-dimensional forward process.

Throughout this section, we set the terminal time to $T = 1$, and discretize the time interval $[0,1]$ into $N_T= 120$ equal subintervals. We employ Adam optimizer for training. To capture the path signature of the common noise $\{W_{[0, t]}^0\}$ with higher accuracy, we sample it with a finer grid with $4N_T$ subintervals. 

To evaluate numerical performance, we define mean absolute error (MAE) at time $t_n$ as: 
\begin{equation}
\text{MAE}_{t_n} = \underset{j}{\text{mean}} \big| m \big(t_n, X_{t_n}^{j},\mc{L} ( X_{t_n}^j\vert \mc{F}_{t_n}^0)(\omega^j) \big)  - \cm \big(t_n, X_{t_n}^j, \mathcal{S}^M(\hat W^{0,j}_{[0,t_n]}) \big) \big|,
\end{equation} 
where the superscript $j \in \{1, \ldots, 10^3\}$ indexes sample trajectory, and $\omega^j$ is the $j^{th}$ realization of the common noise $W^0$. We also define the mean Euclidean error (MEE) for processes as 
\begin{equation}
    \text{MEE}(\theta, \hat{\theta}) = \underset{j,\,t_n}{\text{mean}} \| \theta_{t_n}^{j} - \hat{\theta}_{t_n}^{j} \|,
\end{equation}
where $\theta_{t_n}$ may represent the process $X, Y, Z, Z_0$ or distribution embedding functions $m_i$ valued at time $t_n$, $\hat \theta_{t_n}$ denotes their approximations, and $\| \cdot \|$ is the Euclidean norm. For both MAE and MEE values, we report the mean and standard deviation over five independent training runs.

\subsection{Supervised learning for $m$}\label{sec:numerics1}
Our first example aims to illustrate the efficiency of the supervised learning approach for approximating $m_i$, as proposed in \textbf{Step 2} (see Section~\ref{sec:detailed-algo}). Let $X_t = W_t + W_t^0$, and consider the function
\begin{equation}
    m(t, x, \mc{L}(X_t \vert \mc{F}_t^0)) := \EE_{x_t'\sim \mc{L}(X_t \vert \mc{F}_t^0)}\left[e^{-\|x - x_t'\|^2/q}\right] = \left(\frac{q}{q + 2t}\right)^{q/2} e^{-\frac{\left\Vert x - W_t^0\right\Vert^2}{q+2t}}.
\end{equation}
For each epoch, we sample $N_2 = 2^6$ common noise paths and $N_1 = 2^7$ idiosyncratic noise paths per each common one, resulting  in $N_1N_2$ idiosyncratic paths in total. We perform 300 training epochs, each consisting of 50 stochastic gradient descent (SGD) steps. The initial learning rate is $lr = 0.003$, and is reduced by a factor of 0.95 every 100 SGD steps during the first 100 epochs. Thereafter, $lr$ is kept constant for the remaining epochs.

Our first experiment compares two classes of functions used to approximate $m_i$. \textbf{Archi. 1} searches $\cm$ among linear functionals of the truncated signature, with the functional parameterized by a neural network (NN) taking inputs $(t_n,X_{t_n})$. \textbf{Archi. 2} employs a feedforward NN with inputs $(t_n, X_{t_n}, \mathcal{S}^M(\hat W_{[0, t]}^0))$, where $\mathcal{S}^M(\hat W_{[0, t]}^0)$ is the time-extended signature truncated at order $M$. We evaluate MAE across dimensions $d = q = 1, 5, 10$ and hidden layers $R = 2, 3$, with truncation order fixed at $M = 3$. Results in Figure~\ref{fig: validate 1} show little difference between networks with $R = 2$ or $R=3$, but \textbf{Archi. 2} consistently outperforms \textbf{Archi. 1} under the same hyperparameters. Hence, we adopt \textbf{Archi. 2} for approximating $\cm$ in subsequent experiments.

\begin{figure}[!htbp]
    \centering
    \includegraphics[width=1.0\textwidth]{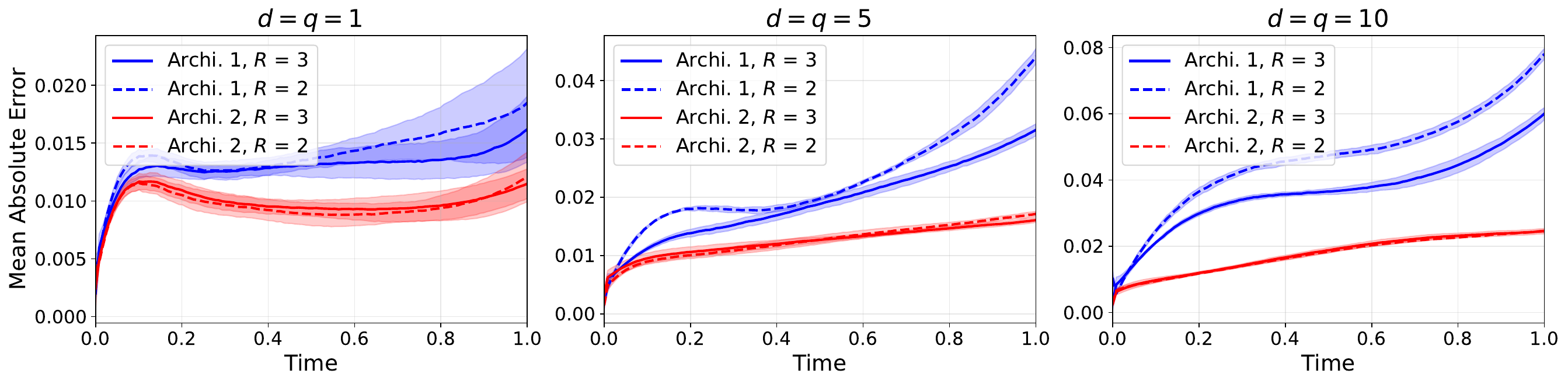} 
    \caption{Mean absolute error (MAE) over time for \textbf{Archi. 1} (linear functional of the truncated signature) and \textbf{Archi. 2} (feedforward neural network with truncated signature input), in dimensions 1, 5, and 10. Solid and dashed lines denote the average MAE across 5 independent runs; shaded areas show $\pm$ standard deviation.}
    \label{fig: validate 1}
\end{figure}

We next examine the effect of the signature truncation order $M$, evaluating MAE for $M = 2, 3, 4$ and dimensions $d = p = 1, 5, 10$, with results shown in Figure~\ref{fig: validate 2}. Performance worsens as $M$ increases, with $M =2$ achieving the lowest error across all cases.  A likely explanation is the tradeoff: higher-order terms add variance and risk overfitting, while low-order signatures already capture the key structure of Brownian paths. Thus, $M =2$ or 3 provides the best balance of accuracy and robustness for our setting.

\begin{figure}[htbp]
    \centering
    \includegraphics[width=1.0\textwidth]{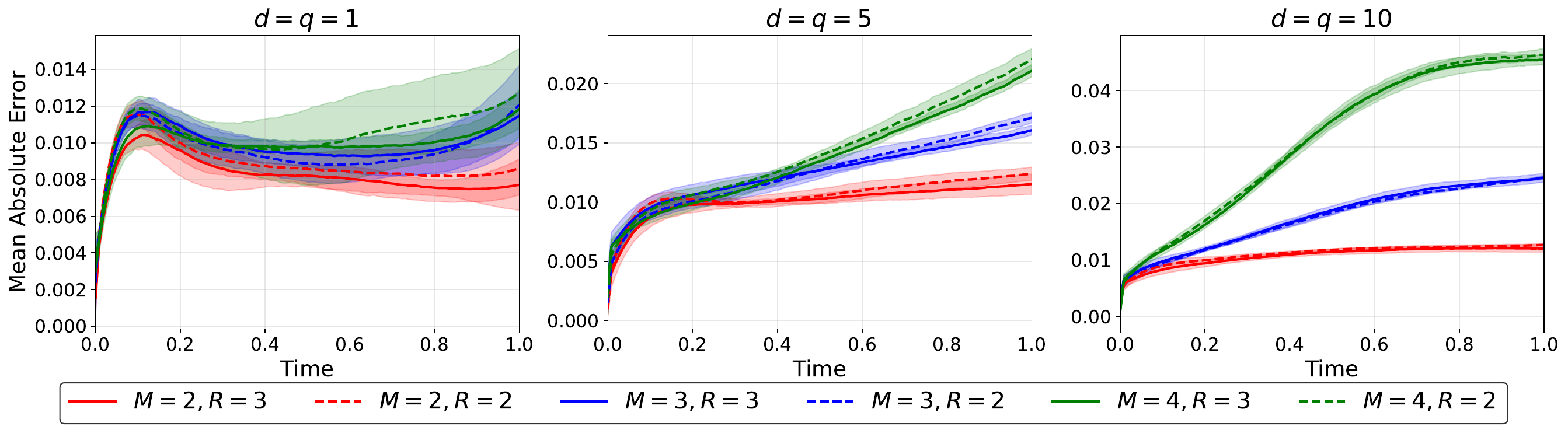} 
    \caption{Mean absolute error (MAE) over time for signature truncation orders $M=2,3,4$ in dimensions $d=p = 1, 5, 10$. Solid and dashed lines represent the average MAE across 5 independent runs; shaded areas indicate $\pm$ one standard deviation. Lower truncation orders yield better performance, with $M=2$ achieving the lowest error.}\label{fig: validate 2}
\end{figure}

To improve scalability in high-dimensional settings, we examine the performance of \textbf{log-signature} transforms \cite{kidger2021signatory, liao2019learning} (cf. Definition \ref{def: log-sig}). Table~\ref{tab:results of supervised learning} reports $\text{MAE}$ for dimensions $d = p = 5, 10, 12, 15$ using \textbf{Archi. 2} with both signature and log-signature representations at truncation orders  $M = 2$ and $M = 3$. Comparing the two truncation levels, we find that log-signatures are less sensitive to the truncation order in terms of MAE, independent of network architecture, with the advantage becoming more pronounced as dimension increases. This highlights the efficiency of the log-signature’s compressed, non-redundant representation, with clear advantages in higher-order and higher-dimensional settings where standard signatures become less efficient.

\begin{table}[htbp]
  \centering
  \begin{tabular}{|c|c|c|c|c|}
    \hline
    \multirow{2}{*}{\textbf{Configuration}} & \multicolumn{2}{c|}{\textbf{$M=2$}} & \multicolumn{2}{c|}{\textbf{$M=3$}} \\
    \cline{2-5}
     & \textbf{log-sig} & \textbf{sig} & \textbf{log-sig} & \textbf{sig} \\
    \hline
    $d = 5, H = 2$ & 9.58e-3 (6.84e-4) & 1.05e-2 (3.76e-4) & 1.03e-2 (3.74e-4) & 1.25e-2 (4.45e-4)\\
    \hline
    $d = 5, H = 3$ & 9.04e-3 (3.10e-4) & 1.00e-2 (3.39e-4) & 1.02e-2 (2.34e-4) & 1.25e-2 (3.70e-4)\\
    \hline 
    $d = 10, H = 2$ & 9.61e-3 (4.07e-4) & 1.11e-2 (2.25e-4) & 1.15e-2 (1.70e-4) & 1.74e-2 (2.22e-4)\\
    \hline 
    $d = 10, H = 3$ & 9.56e-3 (3.13e-4) & 1.07e-2 (3.57e-4) & 1.14e-2 (3.36e-4) & 1.77e-2 (4.76e-4)\\
    \hline
    $d=12$, $H=2$ & 9.80e-3 (2.67e-4) & 1.16e-2 (1.67e-4) & 1.22e-2 (1.25e-4) & 2.33e-2 (5.70e-4) \\
    \hline
    $d=12$, $H=3$ & 9.62e-3 (1.57e-4) & 1.12e-2 (3.53e-4) & 1.23e-2 (3.58e-4) & 2.47e-2 (9.69e-4) \\
    \hline
    $d=15$, $H=2$ & 9.83e-3 (2.29e-4) & 1.23e-2 (4.81e-4) & 1.62e-2 (4.79e-4) & 3.73e-2 (1.92e-3) \\
    \hline
    $d=15$, $H=3$ & 9.76e-3 (2.38e-4) & 1.19e-2 (3.05e-4) & 1.59e-2 (6.14e-4) & 3.49e-2 (8.11e-4) \\
    \hline
  \end{tabular}
  \caption{Time-averaged performance measured by MAE of \(M\) for dimensions \(d=5, 10, 12, 15\). Results are aggregated over $10^3$ independent trajectories. Each entry reports the mean MAE with standard deviation (in parentheses), computed from 5 repeated experiments with different random seeds.}
  \label{tab:results of supervised learning}
  \end{table}

Across all experiments, we find that network architecture has little impact on MAE. Error growth with respect to dimension is moderate overall, though more pronounced for \textbf{Archi. 1} and for higher-order signatures ($M = 3, 4$). These results confirm the effectiveness of neural networks for approximating $m$, particularly in high-dimensional settings.

\subsection{A MV-FBSDE in random environment}
Next, we consider the following MV-FBSDE in random environment for $(X_t, Y_t, Z_t, Z_t^0)$:
\begin{equation}
    \begin{dcases}
        \ud X_t^i &= \bigg[\sin\Big(\EE_{x_t' \sim \mc{L}(X_t \vert \mc{F}_t^0)} e^{- \frac{\|X_t - x_t'\|^2}{d}} - e^{-\frac{\|X_t - \EE[X_t \vert \mc{F}_t^0]\|^2}{d+2t}}\big(\frac{d}{d+2t}\big)^{\frac{d}{2}}\Big) \\
        & \qquad + \frac{1}{2} \Big(\EE[Y_t \vert \mc{F}_t^0] - \sin \big(t + \frac{1}{\sqrt d} \sum_{i=1}^d\EE[X_t^i \vert \mc{F}_t^0]\big)e^{-\frac{t}{2}}\Big) \bigg]\ud t + \ud W_t^i + \ud W_t^{0,i}, \quad 1 \leq i \leq d\\
     \ud Y_t & = \bigg[ \frac{\sum_{i=1}^d (Z_t^i + Z_t^{0, i})}{2 \sqrt d}- Y_t +  \sqrt{2Y_t^2 + \|Z_t\|^2 + \|Z_t^0\|^2 + 1} - \sqrt 3\bigg] \ud t + Z_t \cdot  \ud W_t + Z_t^0 \cdot \ud W_t^0, 
    \end{dcases}
\end{equation}
with initial and terminal conditions $X_0^i = 0$ and 
$Y_T = \sin\big(T + \frac{\sum_{i=1}^d X_T^i}{\sqrt d}\big)$, where $X_t^i$ is the $i^{th}$ entry of the $d$-dimensional process $X_t$. One can check that the solution to the above MV-FBSDE is 
\begin{equation}
    X_t = W_t + W_t^0, \quad Y_t = \sin\Big(t + \frac{\sum_{i=1}^d X_t^i}{\sqrt d}\Big), \quad Z_t^i = Z_t^{0,i} = \frac{1}{\sqrt d} \cos \Big(t + \frac{\sum_{i=1}^d X_t^i}{\sqrt d }\Big).
\end{equation}
The corresponding $m$ functions are
$m_2 \equiv m_3 \equiv m_4 \equiv m_5 \equiv 0$, and 

\begin{equation}
    m_1 = \Big( \tilde \EE \big[ e^{- \frac{\|x - \tilde{x}_t \|^2}{d}}\big], \tilde \EE[\tilde{x}_t], \tilde \EE[\tilde{y}_t]  \Big) \in \mathbb{R}^{1 + d + 1},
\end{equation}
where the expected value $\tilde{\mathbb{E}}$ is with respect to $(\tilde{x}_t, \tilde{y}_t) \sim \mc{L}(X_t, Y_t \vert \mc{F}_t^0)$.

Building on the promising results from the supervised learning of \(m\) in Section \ref{sec:numerics1}, we now evaluate MEE under the following configurations: dimensions \(d = q = 1, 5, 10\), truncation order fixed at \(M = 2\), log-signature as feature representation, and \textbf{Archi. 2}. For hyperparameters, we use \(N_2 = 128\) common noise paths and \(N_1 = 256\)  idiosyncratic noise paths per common noise. The function $m_1$ is approximated by a feedforward NN with $R = 2$ hidden layers of width 64, while $Z_t$ and $Z_t^0$ are approximated by networks with $R = 4$ hidden layers of width 128. 
In the $k$-th fictitious play, the supervised learning of $m_1$ in \textbf{Step 2} is trained for 5000 SGD iterations with initial  $lr = 1.2 \times 10^{-4} \times 0.95^{k-1}$, decayed by a factor of $0.8$ every $500$ steps for stable convergence.
In \textbf{Step 3} (Deep BSDE training), 
$lr = 2.5\times 10^{-4}$ for the first 1000 SGD iterations, then multiplied by 0.3 for the next 500 iterations and again by 0.3 for the last 500 iterations during the first 20 fictitious plays. For subsequent player $(k > 20)$, $lr$ in Step 3 is further reduced by a factor of 0.9 per play. Paths of $(W, W^0)$ are regenerated every 20 iterations.

Results in Table~\ref{tab:results of MV-FBSDE} demonstrate consistent accuracy and convergence of our algorithm. The MEE for all processes $(X_t, Y_t, Z_t, Z_t^0)$ decreases as $k$ increases, with rapid improvement initially and slower gains after about $k=20$ fictitious plays. While the MEE grows and the convergence rate slightly declines with dimension, the overall stable convergence demonstrates the scalability of our framework. Figures~\ref{fig: MAE for dim1 over 5 independent runs.}--\ref{fig: MAE for dim1} present the MEE over time and comparison of between analytical and numerical trajectories for the case $d = q = 5$.

\begin{table}[htbp]
  \centering
  \begin{tabular}{|c|c|c|c|c|c|}
    \hline
    \textbf{$K$} & \textbf{Dim} & \textbf{$X$} & \textbf{$Y$} & \textbf{$Z$} & \textbf{$Z^0$} \\
    \hline
    \multirow{2}{*}{$k=5$}  
      & $d=1$ & 2.73e-2(5.03e-4) & 7.53e-2(2.46e-3) & 5.24e-2(1.01e-3) & 5.90e-2(3.46e-3) \\
      & $d=5$ & 8.21e-2(6.95e-3) & 1.02e-1(7.64e-3) & 6.61e-2(5.85e-3) & 8.38e-2(8.02e-3) \\
      & $d=10$ & 1.21e-1(1.36e-2) & 1.41e-1(8.69e-3) & 1.12e-1(1.57e-2) & 1.17e-1(1.98e-2) \\
      \hline
    \multirow{2}{*}{$k=10$} 
      & $d=1$ & 1.12e-2(7.15e-4) & 4.57e-2(1.04e-3) & 3.19e-2(2.73e-3) & 3.16e-2(5.99e-3) \\
      & $d=5$ & 3.68e-2(2.24e-3) & 5.84e-2(3.06e-3) & 3.59e-2(4.03e-3) & 4.36e-2(4.34e-3) \\
      & $d=10$ & 8.43e-2(1.17e-2) & 8.48e-2(7.08e-3) & 5.92e-2(5.41e-3) & 6.32e-2(1.28e-2) \\
    \hline
    \multirow{2}{*}{$k=20$} 
      & $d=1$ & 8.64e-3(7.81e-4) & 4.12e-2(1.28e-3) & 1.92e-2(3.62e-3) & 2.53e-2(4.41e-3) \\
      & $d=5$ & 2.86e-2(4.22e-3) & 5.03e-2(2.02e-3) & 3.08e-2(2.83e-3) & 3.47e-2(4.92e-3) \\
      & $d=10$ & 5.99e-2(5.76e-3) & 6.49e-2(6.15e-3) & 4.49e-2(2.55e-3) & 4.63e-2(5.68e-3) \\
    \hline
    \multirow{2}{*}{$k=30$} 
      & $d=1$ & 9.28e-3(6.61e-4) & 4.13e-2(1.78e-3) & 1.99e-2(1.79e-3) & 2.20e-2(2.32e-3) \\
      & $d=5$ & 2.43e-2(1.83e-3) & 4.71e-2(9.21e-4) & 2.19e-2(1.81e-3) & 3.08e-2(2.76e-3) \\
      & $d=10$ & 5.32e-2(3.58e-3) & 6.03e-2(1.55e-3) & 3.92e-2((3.53e-2) & 3.98e-2(4.14e-3) \\
      \hline
  \end{tabular}
  \caption{Mean Euclidean Errors (MEE) of the trajectories \( (X_t, Y_t, Z_t, Z_t^0) \) for $d=q=1, 5$ and $10$. Values show mean and standard deviation over 5 independent runs. Performance improves as Fictitious Play (FP) iterations increase. }
  \label{tab:results of MV-FBSDE}
\end{table}

\begin{figure}[!htbp]
    \centering
        \includegraphics[width=0.7\textwidth]{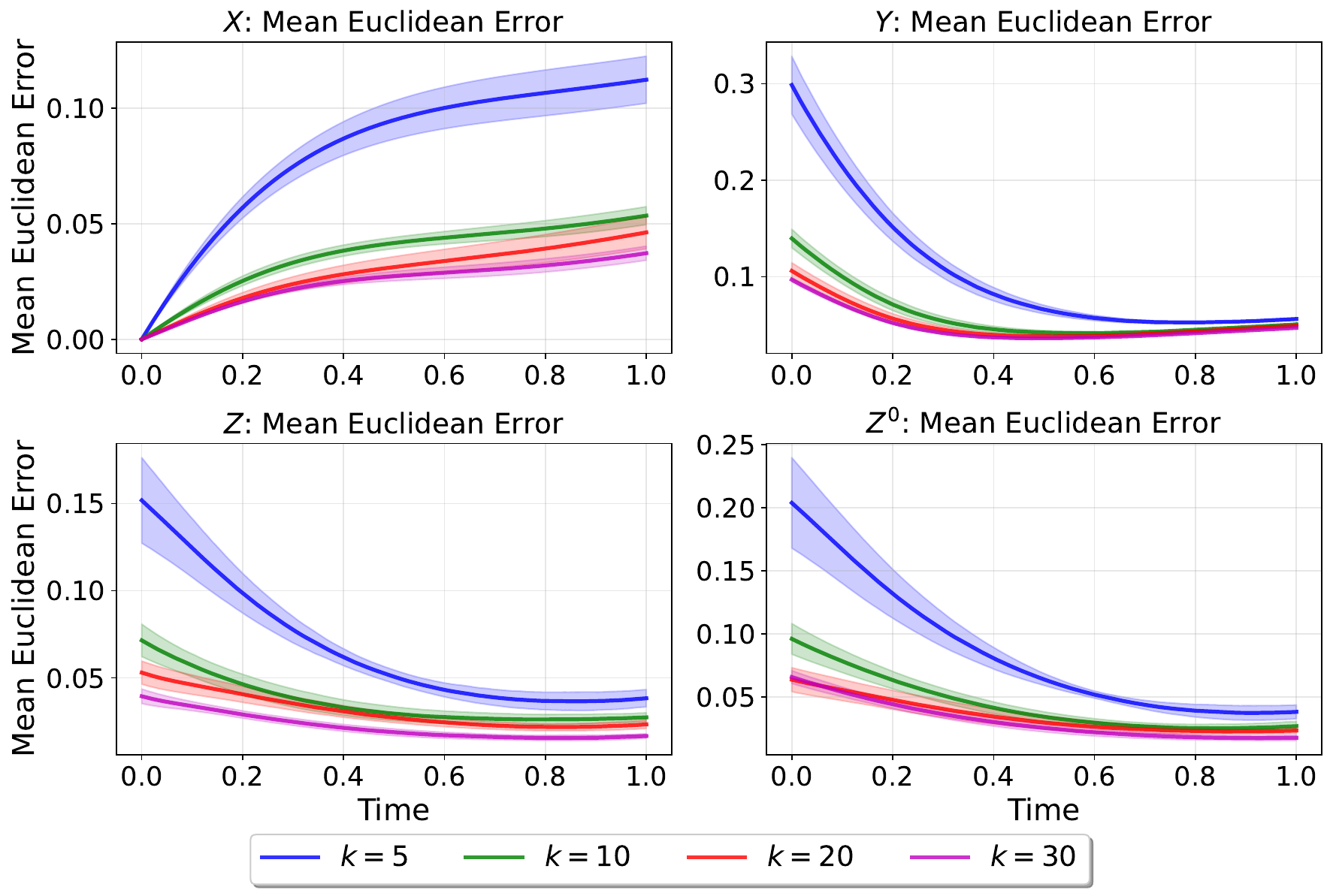}
    \caption{Mean Euclidean Error (MEE) over time for $d = p = 5$ with varying numbers of fictitious play iterations. Shaded regions represent $\pm 1$ standard deviation over five independent runs.}
    \label{fig: MAE for dim1 over 5 independent runs.}
\end{figure}

\begin{figure}[!htbp]
    \centering
    \includegraphics[width=0.7\textwidth]{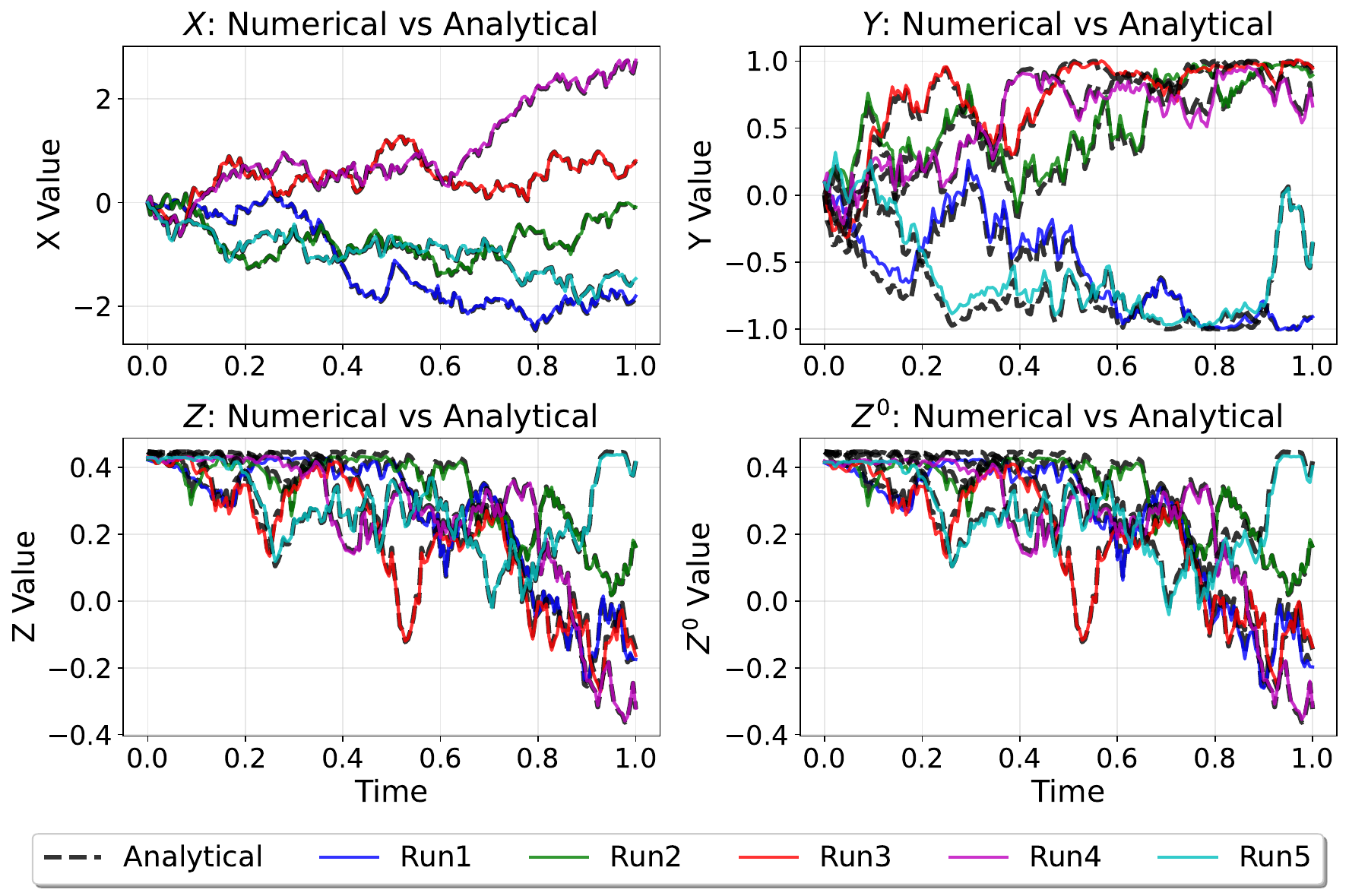} 
    \caption{Comparison of analytical solutions (dashed line) and numerical solutions (solid lines) for $d = q = 5$ obtained after 30 fictitious play rounds over five independent runs.}
    \label{fig: MAE for dim1}
\end{figure}

\subsection{A flocking model with common noise}

In this final example, inspired by a model of Cucker and Smale~\cite{cucker2007emergent}, we consider a flocking model in a mean-field game setting analogous to the one proposed in~\cite{nourian2010synthesis,hanhulong2022learning} but here with common noise. Flocking models have been widely used to describe collective motion in groups of self-propelled agents, such as birds, fish, bacteria, and insects. In our formulation, each agent chooses her acceleration to minimize a cost functional that penalizes both acceleration and misalignment in position and velocity.

The presence of common noise, interpretable as wind for birds or the sudden appearance of a predator for fish, introduces additional complexity, as the problem now depends on the full distribution of the population condition on the common noise. This is significantly more challenging than dependence on conditional moments and, to the best of our knowledge, has not been numerically addressed in the existing literature. We note that the version of the mean-field game without common noise, along with its numerical study, has been presented in \cite{hanhulong2022learning}.

We consider the dynamics of a representative agent whose position \( x_t \in \mathbb{R}^d \) and velocity \( v_t \in \mathbb{R}^d \) evolve according to the controlled stochastic system
\begin{equation}
    \begin{dcases}
      \ud x_t = v_t \, \ud t, \\
      \ud v_t = u_t \, \ud t + C \, \ud W_t + D \, \ud W_t^0,
    \end{dcases}
\end{equation}
where \( u_t \in \mathbb{R}^d \) is the control representing acceleration, \( W_t \) is a \( q \)-dimensional Brownian motion modeling idiosyncratic noise, and \( W_t^0 \) is a \( q \)-dimensional Brownian motion that introduces common noise shared across all agents.

The agent aims to minimize
\begin{equation}\label{eq:CSflock_cost_rewrite}
    \EE \int_0^T (\|u_t\|_R^2 + \mathcal{C}(x_t, v_t; f_t)) \, \ud t,
\end{equation}
where \( \|u_t\|_R^2 := u_t^\top R u_t \) penalizes large acceleration, and the second term, \( \mathcal{C}(x_t, v_t; f_t) \), measures the misalignment between the agent’s position and velocity from a given distribution \( f_t \):
\begin{equation}
    \mathcal{C}(x, v; f) = \left\| \int_{\mathbb{R}^{2n}} w(\|x - x'\|)(v' - v) \, f(x', v') \, \ud x' \ud v' \right\|_Q^2 = \left\| \tilde{\EE}_{(x',v') \sim f} \left[w(\|x - x'\|)(v' - v)\right] \right\|_Q^2.
\end{equation}
Here, the interaction weight function is defined as \( w(x) := (1 + x^2)^{-\beta} \) for some \( \beta \geq 0 \), and \( Q, R \) are symmetric positive definite matrices defining the respective norms \( \|x\|_Q := (x^\top Q x)^{1/2} \) and \( \|u\|_R := (u^\top R u)^{1/2} \). The expectation $\tilde \EE$ is with respect to the distribution $f$. 

To determine a mean-field equilibrium under common noise, one must find an optimal control strategy \( \hat{\mu}_t \) such that \( f_t \) corresponds to the conditional density of the optimal trajectory \( (\hat{x}_t, \hat{v}_t) \) given the common noise filtration, i.e., $f_t \sim \mc{L}(\hat{x}_t, \hat{v}_t \vert \mc{F}_t^0)$.

\subsubsection{The Reformulation through MV-FBSDEs}

Following \cite[Chapter 2]{CarmonaDelarue_book_II}, we characterize the mean-field equilibrium through MV-FBSDEs in random environment with state $X_t = (x_t,v_t)^\top$:
\begin{equation}
\label{eq:CSFlocking}
\begin{dcases}
    \ud x_t = v_t \ud t, \quad \ud v_t = -\frac{1}{2} R^{-1} Y_t^2 \ud t + C \ud W_t + D \ud W_t^0, \quad &(x_0, v_0) = X_0,  \\
   \ud Y_t = -\begin{pmatrix}\partial_x H\\ \partial_v H \end{pmatrix}(t, x_t, v_t, \mc{L}(x_t, v_t \vert \mc{F}_t^0), Y_t, \hat u_t) \ud t + Z_t \ud W_t + Z_t^0 \ud W_t^0, \quad &Y_T = 0,
    \end{dcases}
\end{equation}
where $Y_t \in \mathbb{R}^{2d}$ is the backward process with $Y_t^1 \in \mathbb{R}^d, Y_t^2 \in \mathbb{R}^d$ being the first half and second half entries, $Z_t = \begin{pmatrix}Z_t^1 \\ Z_t^2\end{pmatrix} \in \mathbb{R}^{2d\times q}$ is the adjoint process with $Z_t^1$ and $Z_t^2$ being $\mathbb{R}^{d\times q}$-valued, $Z_t^0 = \begin{pmatrix}Z_t^{0,1} \\ Z_t^{0,2}\end{pmatrix} \in \mathbb{R}^{2d\times q}$ is the adjoint process with $Z_t^{0,1}$ and $Z_t^{0,2}$ being $\mathbb{R}^{d\times q}$-valued. The Hamiltonian $H$ is defined by
\begin{equation}
    H(t, x, v, f, y, u) = (v^\top, u^\top) y +\mathcal{C}(x, v; f)  + \|u\|_R^2,
\end{equation} 
whose unique minimizer is 
\begin{equation}
    \hat u = -\frac{1}{2} R^{-1} y_{d+1:2d},
\end{equation}
where $y \in \mathbb{R}^{2d}$ with $y_{1:d}$ being the first half entries of this vector and $y_{d+1:2d}$ being the second half.

With straightforward computation, the derivatives of $H$ with respect to the state variables $(x, v)$ are:
\begin{align}
    \partial_x H & = \partial_x \mathcal{C}(x, v; f) = 2\tilde \EE_{(x',v') \sim f}[\partial_xw(\|x - x'\|)(v' - v)]^\top Q \tilde \EE_{(x',v')\sim f} [w(\|x - x'\|)(v' - v)], \\
    \partial_v H & = y_{1:n} + \partial_v \mathcal{C}(x, v; f) = y_{1:n} + 2Q \tilde \EE_{(x',v')\sim f} [w(\|x - x'\|)(v' - v)] \tilde \EE_{(x',v') \sim f}[-w(\|x - x'\|)],
\end{align}
where $\partial_xw(\|x - x'\|)(v' - v)$ is understood as a Jacobian matrix
$$
\begin{bmatrix}
 \frac{\partial w(\|x-x'\|)}{\partial x_1}(v_1' - v_1) & \cdots &  \frac{\partial w(\|x-x'\|)}{\partial x_n}(v_1' - v_1) \\
\vdots & \cdots & \vdots \\
 \frac{\partial w(\|x-x'\|)}{\partial x_1}(v_n' - v_n) & \cdots &  \frac{\partial w(\|x-x'\|)}{\partial x_n}(v_n' - v_n)
\end{bmatrix},
$$
and $x_i$ (resp. $v_i$) denotes the $i^{th}$ entry of $x$ (resp. $v$).

Therefore, the distribution dependence functions are identified as $m_1 \equiv = m_2 \equiv m_3 \equiv m_5 \equiv 0$, and 
\begin{align}
m_4(t, x_t, v_t, \mc{L}(x_t, v_t \vert \mc{F}_t^0)) = \left[
\begin{aligned}
&\tilde\EE[\partial_xw(\|x_t - x'_t\|)(v'_t - v_t)]^\top Q \tilde \EE [w(\|x_t - x'_t\|)(v'_t - v_t)]\\
&\tilde \EE [w(\|x_t - x'_t\|)(v'_t - v_t)] \tilde \EE[-w(\|x_t - x'_t\|)]
\end{aligned}
\right],
\end{align}
where $\tilde \EE$ denotes expectation with respect to  $(x_t', v_t')$ distributed according to $\mc{L}(x_t, v_t \vert \mc{F}_t^0)$.

By the stochastic maximum principle, once \eqref{eq:CSFlocking} is solved and the processes $(x_t, v_t,  Y_t,  Z_t, Z_t^0)$ are identified, the optimal control $\hat u_t = -\frac{1}{2} R^{-1} Y_t^2$, together with $\hat f_t = \mc{L}( x_t, v_t \vert \mc{F}_t^0)$ gives a mean-field game equilibrium.

\subsubsection{Analytical benchmark: A special LQ case when $\beta = 0$}

In general, the system~\eqref{eq:CSFlocking} does not admit an explicit analytical solution. However, a tractable case arises when $\beta = 0$ and $w(x) \equiv 1$, in which the problem reduces to a linear-quadratic (LQ) mean-field game with common noise. In this setting, the mean-field cost-coupling function simplifies to
\begin{equation}
    \mathcal{C}(x,v; f) = \left\Vert \tilde \EE_{(x',v')\sim f} [v' - v]\right\Vert_Q^2 = \left\Vert\EE[v \vert \mc{F}_t^0] - v\right\Vert_Q^2.
\end{equation}
The corresponding Hamiltonian derivatives become
\begin{equation}
    \partial_x H = 0, \quad \partial_v H = y_{1:d} + 2Q(v - \tilde \EE_{(x', v')\sim f}[v']) = y_{1:d} + 2Q(v - \EE[v\vert \mc{F}_t^0]).
\end{equation}

Now, the MV-FBSDEs in random environment derived from the stochastic maximum principle reduce to
\begin{equation}\label{eq:CSFlocking_special}
\begin{dcases}
    \ud x_t = v_t \ud t, \quad  \ud v_t = -\frac{1}{2} R^{-1} Y_t^2 \ud t + C \ud W_t + D \ud W_t^0, \quad &(x_0, v_0) = \xi,  \\
   \ud Y_t = -\begin{pmatrix}0\\ Y_t^1 + 2Q(v_t - \EE[v_t \vert \mc{F}_t^0]) \end{pmatrix} \ud t + Z_t \ud W_t + Z_t^0 \ud W_t^0, \quad & Y_T = 0.
\end{dcases}
\end{equation}
It immediately follows that $Y_t^1 \equiv Z_t^1 \equiv  Z_t^{0,1} \equiv 0$. Therefore, it suffices to solve the reduced forward-backward system for $(v_t, Y_t^2, Z_t^2, Z_t^{0,2})$.

Taking conditional expectations with respect to $\mc{F}_t^0$ of the backward equation in~\eqref{eq:CSFlocking_special} gives
\begin{equation}
    \EE[Y_t^2 \vert \mc{F}_t^0] = - \EE\bigg[\int_t^T Z_s^{0,2} \ud W_s^0\bigg\vert \mc{F}_t^0\bigg] = 0.
\end{equation}
Similarly, the forward equation yields 
\begin{equation}\label{eq:CSFlocking_mvt}
 \EE[v_t \vert \mc{F}_t^0] = \EE[v_0 \vert \mc{F}_t^0]-\int_0^t \frac{1}{2} R^{-1}\EE[Y_s^2 \vert \mc{F}_t^0] \ud s + D W_t^0 = \EE[v_0] + D W_t^0,
\end{equation}
where the second term vanishes due to $\EE[Y_s^2 \vert \mc{F}_s^0] = \EE[Y_s^2 \vert \mc{F}_t^0]$ = 0. We then propose the ansatz:
\begin{equation}
    Y_t^2 = \eta(t) (v_t - \EE[v_t \vert \mc{F}_t^0]), \quad \eta(T) = 0,
\end{equation}
where $\eta(t) \in \mathbb{R}^{d \times d}$. Applying Itô’s formula to the ansatz and using~\eqref{eq:CSFlocking_mvt} gives
\begin{equation}
    \ud Y_t^2 = \big[\dot\eta(t)(v_t -\EE[v_t \vert \mc{F}_t^0]) 
    - \frac{1}{2}\eta(t)R^{-1}\eta(t)(v_t - \EE[v_t \vert \mc{F}_t^0]) \big] \ud t + \eta(t) C \ud W_t.
\end{equation}
Comparing this with the backward equation in~\eqref{eq:CSFlocking_special}, we identify $Z_t^{0,2} \equiv 0$,
and deduce the following Riccati equation for $\eta(t)$:
\begin{equation}
        \dot\eta(t) - \frac{1}{2} \eta(t) R^{-1} \eta(t) + 2Q = 0, \quad \eta(T) = 0.
\end{equation}
From this, we further conclude that $Z_t^2 = \eta(t)C$, and the optimal control is given by
\begin{equation}
    \hat u_t = -\frac{1}{2}R^{-1} \eta(t)(v_t - \EE[v_t \vert \mc{F}_t^0]).
\end{equation}

The above ODE for $\eta(t)$ admits an analytical solution, which serves as a benchmark for the numerical experiments presented in the subsequent section when $\beta = 0$. In particular, when $R = 0.5I_d$ and $Q = 0.5I_d$, the solution is $\eta(t)=\frac{e^{2T}-e^{2t}}{e^{2t}+e^{2T}} I_d$,

and the forward SDE becomes
\begin{equation}\label{eq:analytical}
    \begin{cases}
    \ud x_t = v_t \ud t, 
    \\
    \ud v_t = -\frac{1}{2} R^{-1} \eta(t)(v_t - D W_t^0 -\mathbb E[v_0] )\ud t + C \ud W_t + D\ud W_t^0, \quad &(x_0, v_0) = \xi.
    \end{cases}
\end{equation}

\subsubsection{Numerical Results}

We implement the proposed algorithm and evaluate its numerical performance on the flocking model. Following the setup in \cite{hanhulong2022learning}, we set $d = q = 3$ and take $C = 0.1 I_3$ and $D = 0.3I_3$, where $I_3$ denotes $3\times 3$ identity matrix. The SDE is initialized with $x_0 \sim \mathcal{N}(0, I_3)$ and $x_1 \sim \mathcal{N}(1, I_3)$. To approximate $m_4,$ we sample $N_2 = 2^7$ common noise paths and $N_1 = 2^9$ idiosyncratic noise paths per common noise, adopting \textbf{Archi. 2} as described in Section~\ref{sec:numerics1}. 
The approximation of $m_4$ and $y_0$ is carried out using feedforward neural networks with two hidden layers of width 64, employing SiLU and Tanh activations, respectively. 
For $Z$ and $Z_0$, we adopt larger architectures consisting of four hidden layers of width 128 with Tanh activation. The fictitious play procedure is run for 20 stages. In the $k$-th fictitious play, \textbf{Step 2} (supervised learning for $m_4$) is trained for 5000 iterations with an initial $lr = 2.5 \times 10^{-4}$, decayed by a factor of 0.8 every 1000 iterations. \textbf{Step 3} (Deep BSDE training) then runs for 3000 iterations with an initial $lr = 7.5\times 10^{-5}$, reduced by a factor of 0.3 every 1000 iterations. Paths of $(W, W^0)$ are regenerated every 30 iterations.

\begin{figure}[!htbp]
    \centering
        \includegraphics[width=0.7\textwidth]{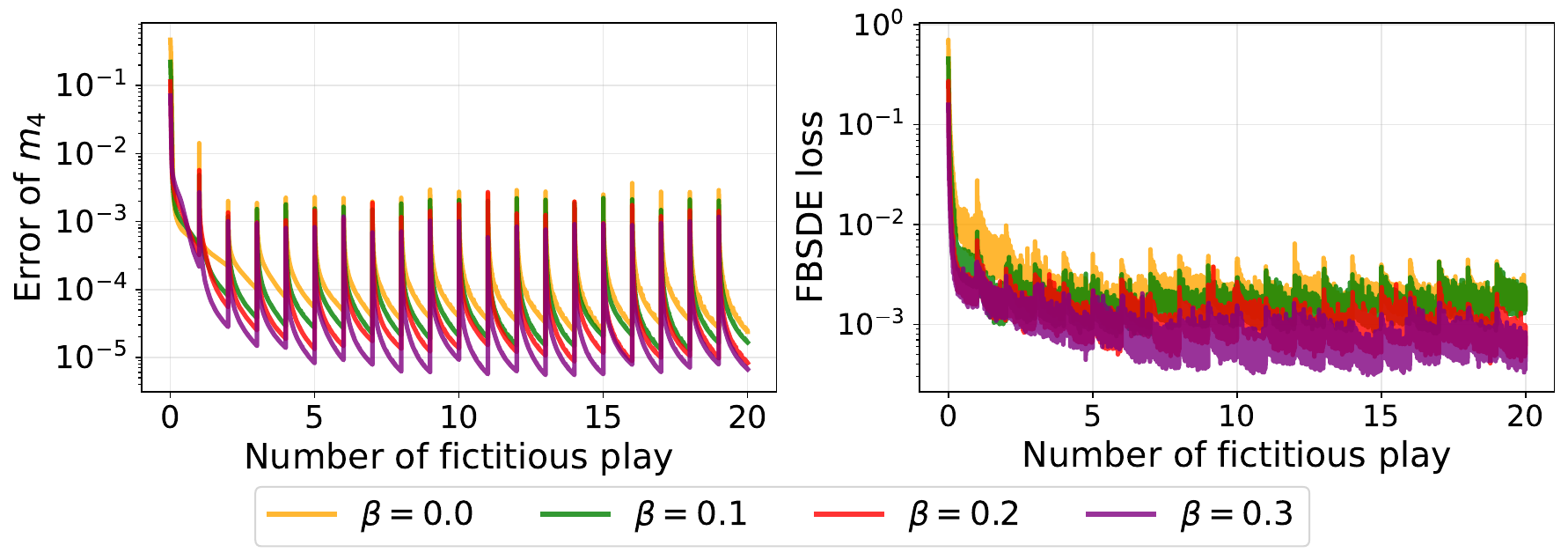}
    \caption{Training loss for Cucker-Smale model with different $\beta$. Left: the loss of supervised learning of $m_4$ in \textbf{Step 2}. Right: the loss of Deep BSDE in \textbf{Step 3}.}
    \label{fig: flock_training_loss}
\end{figure}

Figure~\ref{fig: flock_training_loss} shows the training loss curves. After several fictitious-play iterations, the mean square errors of $m_4$ and the FBSDE loss start below $10^{-2}$ at the beginning of each stage. The error of $m_4$ then decreases rapidly, while the FBSDE loss remains relatively stable within each stage. The behavior is consistent across different values of $\beta$, demonstrating the robustness of the training procedure.

Figure~\ref{fig: conditional-density-comparison} compares the conditional state distributions at time $T$ under two different common noise paths (top and bottom), each with multiple values of $\beta$. The numerical trend for different $\beta$ is consistent with that in \cite{hanhulong2022learning}: larger values of $\beta$ produce more dispersed densities, reflecting weaker misalignment in both position and velocity. The figure also highlights the dependence on the common noise: different realizations substantially alter both the mean and variance of the distribution.

Figure~\ref{fig: ana_num_comparison} presents multiple state trajectories for $\beta = 0$ and compares them with the analytical solution \eqref{eq:analytical}, driven either by different (top) or identical (bottom) common noise paths. Trajectories conditioned on the same common noise exhibit lower volatility, consistent with the stochastic model.

In both Figures~\ref{fig: conditional-density-comparison} and \ref{fig: ana_num_comparison}, the numerical solutions with $\beta = 0$ closely match the analytical results, both in trajectory evolution and in conditional densities. The conditional expectation $\mathbb{E}[v_T \vert \mathcal{F}_T^0]$ is also well-captured. This consistency validates the effectiveness of the learned network in approximating the $m_i$ functions and FBSDE solutions for the Cucker-Smale MFG system.

\begin{figure}[!htbp]
    \centering
        \includegraphics[width=0.7\textwidth]{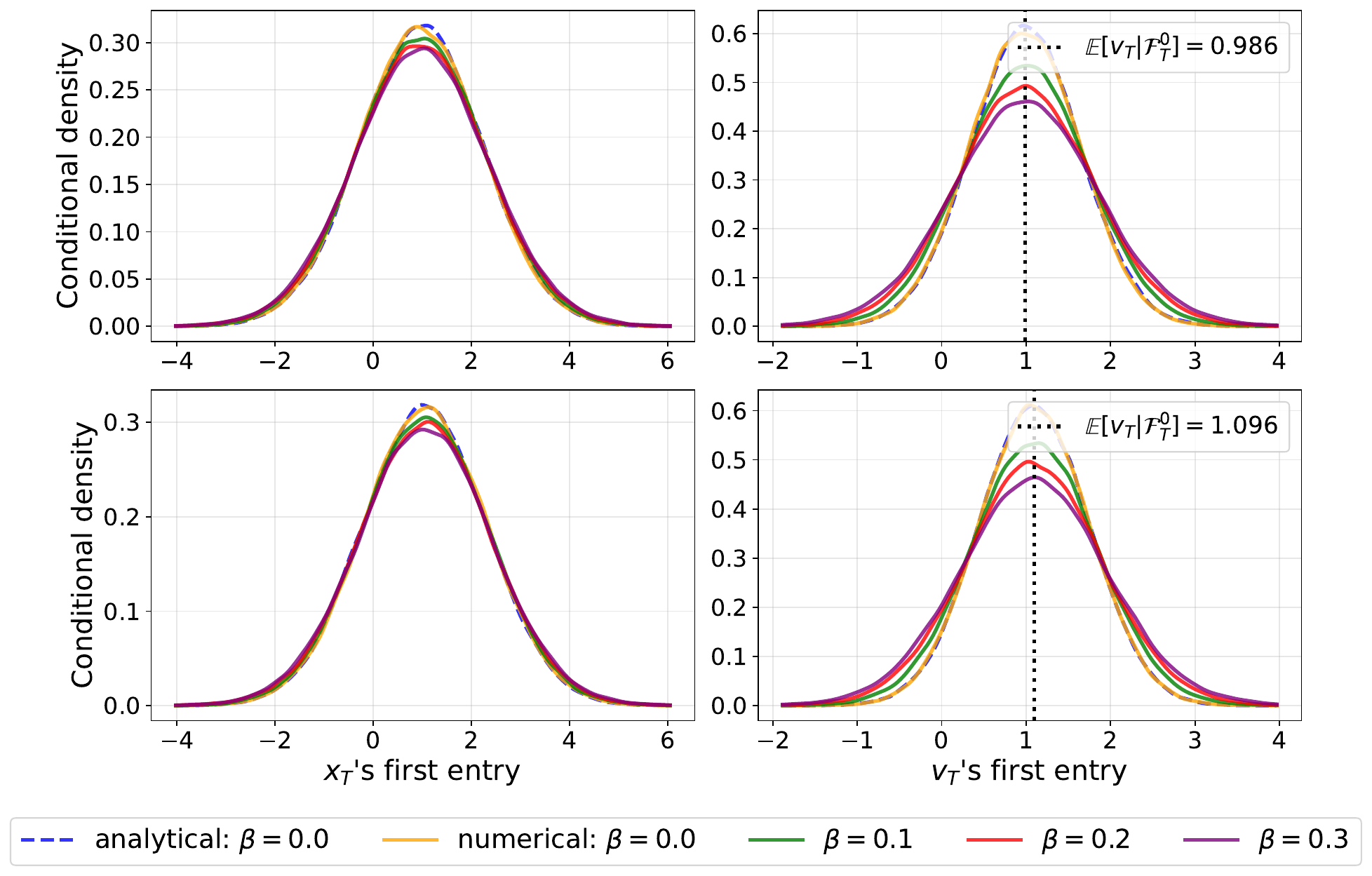}
    \caption{Conditional density of the Cucker-Smale MFG system's final states on two independent realizations of the common noise paths. Top: first noise realization; Bottom: second noise realization. Left: conditional density of position $x_T$; right: conditional density of velocity $v_T$. The dotted line in the velocity plots marks the conditional expectation derived from (\ref{eq:CSFlocking_mvt}).} 
    \label{fig: conditional-density-comparison}
\end{figure}

\begin{figure}[!htbp]
    \centering
        \includegraphics[width=0.7\textwidth]{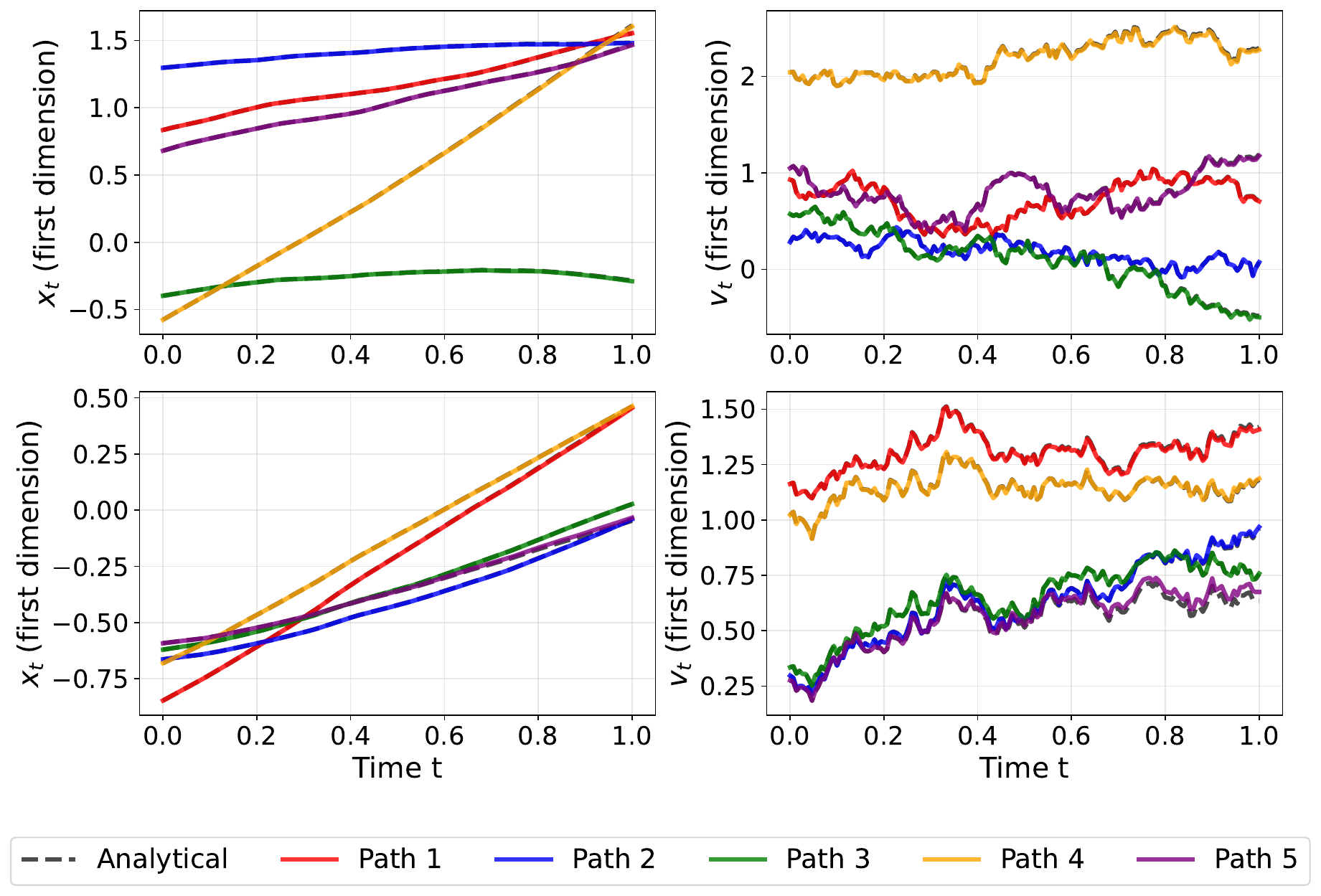}
    \caption{Comparison of analytical solution~\eqref{eq:analytical} and numerical solutions in the LQ case ($\beta = 0$). Top: unconditional trajectories with five independent realizations of both idiosyncratic and common noise. Bottom: conditional trajectories with five independent idiosyncratic noise realizations under a single realization of common noise.}
    \label{fig: ana_num_comparison}
\end{figure}

\bibliographystyle{plain}
\bibliography{refs}

\end{document}